\documentclass[aap,preprint]{imsart}

\RequirePackage{amsthm,amsmath,amsfonts,amssymb}

\RequirePackage[round,authoryear]{natbib}

\RequirePackage[colorlinks,citecolor=blue,urlcolor=blue]{hyperref}
\RequirePackage{graphicx}
\usepackage{bbm,bm}
\usepackage{enumerate}
\usepackage{caption,subcaption}
\usepackage{xcolor}
\usepackage{tikz}
\usepackage{graphics}
\usepackage{appendix}
\usetikzlibrary{arrows}
\newcommand{\midarrow}{\tikz \draw[-triangle 90] (0,0) -- +(.1,0);}

\numberwithin{equation}{section}
\numberwithin{figure}{section}

\startlocaldefs
\theoremstyle{plain}

\newtheorem{assumption}{Assumption}[]
\newtheorem{theorem}{Theorem}[section]
\newtheorem{lemma}[theorem]{Lemma}

\newtheorem{proposition}[theorem]{Proposition}
\theoremstyle{remark}
\newtheorem{remark}{Remark}[section]

\newcommand{\bbA}{\mathbf{A}}

\newcommand{\bbB}{\mathbf{B}}

\newcommand{\bbC}{\mathbf{C}}

\newcommand{\bbD}{\mathbf{D}}
\newcommand{\bbM}{\mathbf{M}}
\newcommand{\bbX}{\mathbf{X}}

\newcommand{\bbr}{\mathbf{r}}
\newcommand{\bbT}{\mathbf{T}}
\newcommand{\bbI}{\mathbf{I}}
\newcommand{\bbu}{\mathbf{u}}
\newcommand{\bbe}{\mathbf{e}}

\newcommand{\tdf}{\tilde{f}}
\newcommand{\tdh}{\tilde{h}}
\newcommand{\bigo}{\mathcal{O}}
\newcommand{\calU}{\mathcal{U}}
\newcommand{\calC}{\mathcal{C}}

\newcommand{\bE}{\mathbb{E}}
\newcommand{\bP}{\mathbb{P}}
\newcommand{\Cov}{{\rm Cov}}
\newcommand{\mtr}{\mathrm{tr}}
\newcommand{\Log}{\mathrm{Log}}

\begin{document}

	\begin{frontmatter}
		\title{Functional CLT for general sample covariance matrices}
		\runtitle{Rates of convergence in CLT for LSS}
		
		\begin{aug}
	
			\author[A]{\fnms{Jian}~\snm{Cui}\ead[label=e1]{cuij836@nenu.edu.cn}},
			\author[B]{\fnms{Zhijun}~\snm{Liu}\ead[label=e2]{liuzj@mail.neu.edu.cn}},
			\author[A]{\fnms{Jiang}~\snm{Hu}\ead[label=e3]{huj156@nenu.edu.cn}}, \and
			\author[A]{\fnms{Zhidong}~\snm{Bai}\ead[label=e4]{baizd@nenu.edu.cn}}

			\address[A]{KLASMOE and School of Mathematics and Statistics, Northeast Normal University, China.\printead[presep={,\ }]{e1,e3,e4}}
			
			\address[B]{College of Sciences, Northeastern University, China.\printead[presep={,\ }]{e2}}
		\end{aug}
		
		\begin{abstract}
		This paper studies the central limit theorems (CLTs) for linear spectral statistics (LSSs) of general  sample covariance matrices, when the test functions belong to $C^3$, the class of functions with continuous third order derivatives. We consider matrices of the form $\bbB_n=(1/n)\bbT_p^{1/2}\bbX_n\bbX_n^*\bbT_p^{1/2},$ where $\bbX_n= (x_{i j} ) $ is a $p \times n$ matrix whose entries are independent and identically distributed (i.i.d.) real or complex random variables, and $\bbT_p$ is a $p\times p$ nonrandom Hermitian nonnegative definite matrix with its spectral norm uniformly bounded in $p$. By using Bernstein polynomial approximation, we show that,  under $\mathbb{E}|x_{ij}|^{8}<\infty$, the centered LSSs of $\bbB_n$ have Gaussian limits. Under the stronger $\mathbb{E}|x_{ij}|^{10}<\infty$, we further establish convergence rates $\bigo(n^{-1/2+\kappa})$ in Kolmogorov--Smirnov $\bigo(n^{-1/2+\kappa})$, for any fixed $\kappa>0$. 		
	
		\end{abstract}
		
		
	\end{frontmatter}

\section{Introduction}
\subsection{Main results}
The central limit theorems (CLTs) for the linear spectral statistics (LSSs) of large-dimensional sample covariance matrices have now been extensively studied due to their significant roles in multiple fields
(e.g. \cite{Jonsson82Lb,BaiS04C,BaiM07A,BaiJ09C,BaiW10F,ZhengB15S,BaoL15S,HuL19H,LiuH23C,BaoH24S}). The objective of the present paper is twofold: first, to establish CLTs for LSSs of  general sample covariance matrices under $C^3$ test function; and second, to quantify the convergence rates of these CLTs in Kolmogorov-Smirnov distance, by means of Bernstein polynomial approximation. In this way, our results complete the functional CLT established in \cite{BaiW10F}.

Consider the \textit{general sample covariance matrix} given by
\begin{align*}
	\bbB_n=\frac{1}{n}\bbT_p^{1/2}\bbX_n\bbX_n^*\bbT_p^{1/2},
\end{align*}
where $\bbX_n$ is a $p\times n$ random matrix with independent and identically distributed (i.i.d.) real or complex entries, and $\bbT_p$ is a $p\times p$ nonrandom Hermitian nonnegative definite matrix. For a given test function $f$, the LSS of $\bbB_n$ is defined as
\begin{align*}
	\sum_{i=1}^pf(\lambda_i^{\bbB_n})=p\int f(x) dF^{\bbB_n}(x),
\end{align*}
where $\lambda_1^{\bbB_n}\geq\dots\geq \lambda_p^{\bbB_n}$ are the ordered eigenvalues of $\bbB_n$, and $$F^{\bbB_n}(x)=\frac{1}{p}\sum_{i=1}^{p}I({\lambda_i^{\bbB_n}}\leq x)$$ 
is the \textit{empirical spectral distribution} (ESD) of $\bbB_n$, where $I(\cdot)$ denotes the indicator function.

A natural starting point in the study of the ESD of $\bbB_n$ is its \textit{limiting spectral distribution} (LSD) of $\bbB_n$. This problem goes back to the seminal work of \cite{MarcenkoP67D}. A major subsequent development was obtained by \cite{Silverstein95S}, which demonstrated that under the conditions that as $n\to\infty$, $y_n=p/n\to y>0$  and
$H_p=F^{\bbT_p}\stackrel{d}{\rightarrow}H$, then  $F^{\bbB_n}\stackrel{d}{\rightarrow}F^{y,H}$ almost surely (a.s.), where "$\stackrel{d}{\rightarrow}$" represents convergence in distribution. A crucial tool in this analysis is the $\textit{Stieltjes transform}$ which is defined by
\begin{align*}
	s_{G}(z)=\int\frac{1}{\lambda-z}dG(\lambda), \quad\quad z\in\mathbb{C}^+,
\end{align*}	
for a given distribution function $G$. The Stieltjes transform of $F^{y,H}$, denoted by $s:=s_{F^{y,H}}(z)$, satisfies the equation:
\begin{align}\label{equation of s}
	s=\int \frac{1}{t(1-y-y z s)-z} d H(t).
\end{align}
Here  $H$ and $F^{y,H}$ are proper distribution functions. 	We refer to $F^{y,H}$ as the LSD of $\bbB_n$.
Since the matrix $\underline{\bbB}_n:=(1/n)\bbX_n^*\bbT_p\bbX_n$ shares the same nonzero eigenvalues as $\bbB_n$, the LSD of $\underline{\bbB}_n$ is given by
\begin{align*}
	\underline{F}^{y,H}=(1-y)\delta_0+yF^{y,H},
\end{align*}
where $\delta_0$ is a Dirac delta function.
Consequently, equation \eqref{equation of s} can be rewritten as 
\begin{align}\label{equation of s^0}
	\underline{s}^0=-\left(z-y\int\frac{t}{1+t\underline{s}^0}dH(t)\right)^{-1},
\end{align}
where $\underline{s}^0:=s_{\underline{F}^{y,H}}$ is the Stieltjes transform of 	$\underline{F}^{y,H}$, and  $s_n^0$ is the Stieltjes transform of $F^{y_n,H_p}$  is the LSD $F^{y,H}$ with parameters $\{y,H\}$ replaced by $\{y_n,H_p\}$.

Once the LSD has been identified, the next natural question is to study the fluctuation of the LSS around its deterministic limit. \cite{BaiS04C} proved that if the fourth moment of $x_{11}$ exists and $f$ is analytic on a certain interval, then under some mild technical conditions, 
\begin{align}\label{LSS tend to normal distribution}
	\int f(x)d G_n(x)\stackrel{d}{\rightarrow} \mathcal{N}(\mu(f),\sigma^2(f)),
\end{align}
where $\mu(f)$ and $\sigma^2(f)$ will be demonstrated in the following Theorem \ref{FunctionCLT}, and
\begin{align*}
	G_n(x)=p\left[F^{\bbB_n}(x)-F^{y_n,H_p}(x)\right].
\end{align*}

Now, we present the main theorems. Prior to this, we introduce the assumptions employed in our paper.

\begin{assumption}\label{assum8th}
	For each $n$, $\bbX_n=(x_{ij})_{p\times n}$, where $x_{ij}$ are i.i.d. with common moments
	\begin{align*}
		&\mathbb{E}x_{ij}=0, \quad \mathbb{E}\lvert x_{ij}\rvert^2=1,  \quad \mathbb{E}\lvert x_{ij}\rvert^{8}<\infty, \\
		& \quad \beta_x=\mathbb{E}\lvert x_{ij}\rvert^4-\lvert\mathbb{E}x_{ij}^2\rvert^2-2, \quad \alpha_x=\lvert\mathbb{E}x_{ij}^2\rvert^2.
	\end{align*}
\end{assumption}

\begin{assumption}\label{assumRDS}
As $min\left\{p,n\right\}\to\infty$,	the ratio of the dimension to sample size (RDS) $y_n=p/n\rightarrow y $, where $y_n\in (0,1)$.
\end{assumption}

\begin{assumption}\label{assumpopulationmatrix}
	$\bbT_p$ is $p\times p$ nonrandom Hermitian nonnegative definite matrix with its spectral norm bounded in $p$, with $H_p=F^{\bbT_p} \stackrel{d}{\rightarrow} H$, a proper distribution function.
\end{assumption}

\begin{assumption}\label{assumtestf}
	Let $f_1,\dots,f_k$ be in $C^3(\calU)$, where $\calU$ denotes any open interval including 
	\begin{align}\label{supportset}
		\left[\liminf_p \lambda_{min}^{\bbT_p}I_{(0,1)}(y)(1-\sqrt{y})^2, \limsup_p\lambda_{max}^{\bbT_p}(1+\sqrt{y})^2 \right],
	\end{align}
	and $C^3(\calU)$ denotes the set of functions $f:\calU\to \mathbb{C}$ which have continuous third order derivatives.
\end{assumption}

\begin{assumption}\label{assumRG}
	The matrix $\bbT_p$ and entries $x_{ij}$ are real.
\end{assumption}

\begin{assumption}\label{assumCG}
	The entries $x_{ij}$ are complex, with $\alpha_x=0$.
\end{assumption}

\begin{assumption}\label{assumhadma}
	As $min\left\{p,n\right\}\to\infty$,
	$$p^{-1} \mtr \left[(\bbT_p(\underline{s}_n^0(z_1)\bbT_p+\bbI)^{-1})\circ (\bbT_p(\underline{s}_n^0(z_2)\bbT_p+\bbI)^{-1})\right] \to h_1(z_1,z_2),$$
	$$p^{-1} \mtr \left[((\underline{s}_n^0(z)\bbT_p+\bbI)^{-2}\bbT_p) \circ ((\underline{s}_n^0(z)\bbT_p+\bbI)^{-1}\bbT_p) \right] \to h_2(z).$$
\end{assumption}

\begin{remark}
	The assumptions on $\bE |x_{ij}|^8<\infty$ and  $y_n\in (0,1)$ are, in fact, highly relevant to Lemma \ref{convergence_rate_ESD}.
\end{remark}

\begin{remark}
	As shown in \cite{ZhengB15S}, there is a counterexample that shows that the convergence  of LSS of sample covariance matrices does not happen when $x_{ij}$ are complex and $\alpha_x\neq 0$. 
\end{remark}

\begin{theorem}\label{FunctionCLT}
	$(\romannumeral1)$ Under Assumptions \ref{assum8th}--\ref{assumRG} and \ref{assumhadma}, then the random vector 
	$$\left(\int f_1(x) d G_n(x),\dots,\int f_k(x) d G_n(x)\right)$$ 
	converges weakly to a Gaussian vector $(X_{f_1},\dots,X_{f_k})$, with mean
	\begin{align}\label{mu}
		\bE X_f=&\frac{1}{2\pi i}\oint_{\calC} f(z) \frac{y\int \underline{s}^0(z)^3t^2(1+t\underline{s}^0(z))^{-3}d H(t)}{(1-y\int \underline{s}^0(z)^2t^2(1+t\underline{s}^0(z))^{-2}d H(t))^2}d z \\
		\nonumber &- \frac{\beta_x}{2\pi i}\oint_{\calC} f(z) \frac{y (\underline{s}^0(z))^3 h_2(z)}{1-y\int \underline{s}^0(z)^2t^2(1+t\underline{s}^0(z))^{-2} d H(t)} d z,
	\end{align}
	and covariance function
	\begin{align}\label{var}
		\Cov(X_f,X_g)=& -\frac{1}{2\pi^2}\oint_{\mathcal{C}_1}\oint_{\mathcal{C}_2}f^{\prime}(z_1)g^{\prime}(z_2)a(z_1,z_2)\int_{0}^{1}\frac{1}{1-ta(z_1,z_2)}d td z_1 d z_2 \\
		\nonumber &-\frac{y\beta_x}{4\pi^2}\oint_{\mathcal{C}_{1}}\oint_{\mathcal{C}_{2}}f^{\prime}(z_1)g^{\prime}(z_2) \underline{s}^0(z_1)\underline{s}^0(z_2)h_1(z_1,z_2)dz_2dz_1,
	\end{align}
	where 
	\begin{align*}
		a(z_1,z_2)=y\underline{s}^0(z_1)\underline{s}^0(z_2)\int\frac{t^2d H(t)}{(1+t\underline{s}^0(z_1))(1+t\underline{s}^0(z_2))}.
	\end{align*}
	The contours $\mathcal{C},\mathcal{C}_1,\mathcal{C}_2$ are closed and oriented in the positive direction in the complex plane, each enclosing the support set defined in \eqref{supportset}. We may assume $\mathcal{C}_1$ and $\mathcal{C}_2 $ to be nonoverlapping.
	
	$(\romannumeral2)$ Under Assumptions \ref{assum8th}--\ref{assumtestf} and \ref{assumCG}, \ref{assumhadma}, then $(\romannumeral1)$ also holds, except the mean is 
	\begin{align}
		\bE X_f= -\frac{\beta_x}{2\pi i}\oint_{\calC} f(z) \frac{y \underline{s}^0(z)^3 h_2(z)}{1-y\int \underline{s}^0(z)^2t^2(1+t\underline{s}^0(z))^{-2} d H(t)} d z
	\end{align}
	 and  the covariance function is 
	 \begin{align}
	 	\Cov(X_f,X_g)=&-\frac{1}{4\pi^2}\oint_{\mathcal{C}_{1}}\oint_{\mathcal{C}_{2}}f^{\prime}(z_1)g^{\prime}(z_2) a(z_1,z_2)\int_{0}^{1}\frac{1}{1-ta(z_1,z_2)}dt dz_2dz_1 \\
	\nonumber & -\frac{y\beta_x}{4\pi^2}\oint_{\mathcal{C}_{1}}\oint_{\mathcal{C}_{2}}f^{\prime}(z_1)g^{\prime}(z_2)  \underline{s}^0(z_1)\underline{s}^0(z_2)h_1(z_1,z_2)dz_2dz_1.	
	 \end{align}
 
\end{theorem}

\begin{remark}
	It is worth noting that the forms of \eqref{mu} and \eqref{var} align with those in \cite{BaiS04C,PanZ08C}. Therefore, this theorem essentially weakens the condition on test function required in \cite{BaiS04C,PanZ08C}. 
\end{remark}

\begin{remark}
	In the course of the proof, we find that, if we replace  $(y,H,\underline{s}^0)$ by $(y_n,H_p,\underline{s}_n^0)$ in \eqref{mu} and \eqref{var}, the convergence rates of the corresponding covariance term and mean term are $\bigo(n^{-1/2})$ and $\bigo(n^{-1})$ separately.
\end{remark}

Let $\mathbb{K}(Y,Z)$ denote the Kolmogorov-Smirnov distance between two real random variables $Y$ and $Z$, defined as 
\begin{align*}
	\mathbb{K}(Y,Z):=\sup_{x\in\mathbb{R}}\lvert \mathbb{P}(Y\leq x)-\mathbb{P}(Z\leq x)\rvert.
\end{align*}
Throughout this paper, $Z$ represents a standard normal random variable and $\Phi$ denotes its distribution function. We now state the theorem on the convergence rates.

\begin{theorem}\label{RateFunctionCLT}
	 Suppose that $\mathbb{E}\lvert x_{ij}\rvert^{10}<\infty$, for any $1\leq i\leq p$, $1\leq j\leq n$.
	 
	 $(\romannumeral1)$ Under Assumptions \ref{assum8th}--\ref{assumRG}, for any fixed $\kappa>0$, there exists a positive constant $K$ independent of $n$, such that 
	 \begin{align*}
	 	\mathbb{K}\left(\frac{\int f(x)d G_n(x)-\mu_n(f)}{\sigma_n(f)},Z\right)=\bigo( n^{-1/2+\kappa}),
	 \end{align*}
	 where
	 \begin{align}\label{meanofn}
	\mu_n(f)=&-\frac{1}{2\pi i}\oint_{\mathcal{C}} f(z) \frac{y_n\int \underline{s}_n^0(z)^3t^2(1+t\underline{s}_n^0(z))^{-3}d H_p(t)}{(1-y_n\int \underline{s}_n^0(z)^2t^2(1+t\underline{s}_n^0(z))^{-2}d 
	H_p(t))^2}d z \\
	\nonumber &-\frac{\beta_x}{2\pi i} \oint_{\mathcal{C}} f(z) \frac{ (\underline{s}_n^0(z))^3 n^{-1} \mtr \left[((\underline{s}_n^0(z)\bbT_p+\bbI)^{-2}\bbT_p) \circ ((\underline{s}_n^0(z)\bbT_p+\bbI)^{-1}\bbT_p) \right]}{1-y_n\int \underline{s}_n^0(z)^2t^2(1+t\underline{s}_n^0(z))^{-2}dH_p(t)} dz
\end{align} 
and
\begin{align}\label{varinceofn}
	\sigma_n^2(f)= &-\frac{1}{2\pi^2}\oint_{\mathcal{C}_1}\oint_{\mathcal{C}_2}f^{\prime}(z_1)f^{\prime}(z_2)a_n(z_1,z_2)\int_{0}^{1}\frac{1}{1-ta_n(z_1,z_2)}d t d z_2 d z_1 \\
	\nonumber & -\frac{\beta_x}{4\pi^2}\oint_{\mathcal{C}_{1}}\oint_{\mathcal{C}_{2}}f^{\prime}(z_1)f^{\prime}(z_2) \underline{s}_n^0(z_1)\underline{s}_n^0(z_2)n^{-1} \\
	\nonumber &\quad\quad\quad\quad  \mtr \left[(\bbT_p(\underline{s}_n^0(z_1)\bbT_p+\bbI)^{-1})\circ (\bbT_p(\underline{s}_n^0(z_2)\bbT_p+\bbI)^{-1})\right] dz_2dz_1,
\end{align}
and
\begin{align*}
	a_n(z_1,z_2)=y_n\underline{s}_n^0(z_1)\underline{s}_n^0(z_2)\int\frac{t^2d H_p(t)}{(1+t\underline{s}_n^0(z_1))(1+t\underline{s}_n^0(z_2))},
\end{align*}
where $\calC$, $\calC_1$ and $\calC_2$ are defined in Theorem \ref{FunctionCLT}.
	 
	 $(\romannumeral2)$ Under Assumptions \ref{assum8th}--\ref{assumtestf} and \ref{assumCG}, $(\romannumeral1)$ also holds, except  
	 \begin{align*}
	     \mu_n(f)=-\frac{\beta_x}{2\pi i} \oint_{\mathcal{C}} f(z) \frac{ (\underline{s}_n^0(z))^3 n^{-1} \mtr \left[((\underline{s}_n^0(z)\bbT_p+\bbI)^{-2}\bbT_p) \circ ((\underline{s}_n^0(z)\bbT_p+\bbI)^{-1}\bbT_p) \right]}{1-y_n\int \underline{s}_n^0(z)^2t^2(1+t\underline{s}_n^0(z))^{-2}dH_p(t)} dz  
	 \end{align*}
	and
	\begin{align*}
		\sigma_n^2(f)= &-\frac{1}{4\pi^2}\oint_{\mathcal{C}_1}\oint_{\mathcal{C}_2}f^{\prime}(z_1)f^{\prime}(z_2)a_n(z_1,z_2)\int_{0}^{1}\frac{1}{1-ta_n(z_1,z_2)}d t d z_2 d z_1 \\
	\nonumber & -\frac{\beta_x}{4\pi^2}\oint_{\mathcal{C}_{1}}\oint_{\mathcal{C}_{2}}f^{\prime}(z_1)f^{\prime}(z_2) \underline{s}_n^0(z_1)\underline{s}_n^0(z_2)n^{-1} \\
	\nonumber &\quad\quad\quad\quad  \mtr \left[(\bbT_p(\underline{s}_n^0(z_1)\bbT_p+\bbI)^{-1})\circ (\bbT_p(\underline{s}_n^0(z_2)\bbT_p+\bbI)^{-1})\right] dz_2dz_1.            
	\end{align*}
	
\end{theorem}

\begin{remark}
	The additional assumption on  $\bE|x_{ij}|^{10}<\infty$ is related to assumptions in \cite{CuiH25R}. 
\end{remark}

\begin{remark}
	Compared to the form  mean and covariance in Theorem \ref{FunctionCLT}, the modifications on the mean and variance terms are motivated by the fact that the convergence rates of $y_n\to y$, $\underline{s}_n^0\to s^0$ and $H_p\stackrel{d}{\rightarrow} H$ remain unknown.
\end{remark}

\subsection{Literature review}
As research on the CLT for LSS of sample covariance matrices has deepened since \cite{BaiS04C}, many publications have increasingly examined the assumptions made for test functions. The earliest article on this field can be traced back to \cite{LytovaP09C}, required the test functions in $C^5$, using the Fourier transform. \cite{Shcherbina11C} further relaxed this condition of test function $f$ to $\lVert f\rVert_{3/2+\kappa}<\infty$, $\kappa>0$, where $\lVert f\rVert^2_{s}$ is the Sobolev $H^s$-norm defined via its Fourier transform $\widehat{f}$:
\begin{align*}
	\lVert f\rVert^2_{\epsilon}=\int (1+2|k|)^{2s}|\widehat{f}(k)|^2 d k, \quad s>3/2.	
\end{align*}
\cite{BaiW10F} weakened the analytic requirement on the test function $f$ from another perspective by employing a Bernstein polynomial approximation, thereby relaxing the condition to $f\in C^4$. It should be noted that \cite{LytovaP09C,BaiW10F,Shcherbina11C} all deal with the sample covariance matrix in the case where $\bbT_p=\bbI_p$. In addition to the approaches mentioned above, another method is based on the Helffer-Sj\"ostrand formula. \cite{NajimY16G} applied this approach to further weaken the requirement to condition $f\in C^3$ for the general sample covariance matrix; \cite{DingW25G} required $f\in C^2_c$ where the subscript $c$ represents that the functions have a compact support, given that $\bbT_p$ is a diagonal matrix.

In the field of random matrix theory (RMT), most existing studies focus primarily on establishing CLTs, whereas relatively few studies investigate the convergence rates of these CLTs. The earliest contributions in this direction concern random matrices on compact group, as discussed in \cite{DiaconisS94E,Johansson97R}. Later, \cite{BerezinB21R} investigated the convergence rate of the CLT for LSS of the Gaussian unitary ensembles (GUE), Laguerre unitary ensembles (LUE), and Jacobi unitary ensembles (JUE) in Kolmogorov-Smirnov distance. More recently, \cite{LambertL19Q} studied the CLT for LSS of $\beta$-ensembles in the one-cut regime, providing a rate of convergence in the quadratic Kantorovich or Wasserstein-2 distance by means of Stein's method;  \cite{BaoH23Q} derived a near-optimal convergence rate for the CLT for LSS of $n\times n$ Wigner matrices, in Kolmogorov-Smirnov distance, which is either $n^{-1/2+\kappa}$ or $n^{-1+\kappa}$ depending on the first Chebyshev coefficient of test function $f$ and the third moment of the diagonal matrix entries; \cite{SchnelliX23C} established a quantitative version of the Tracy-Widom law for the largest eigenvalue of large-dimensional sample covariance matrices, demonstrating convergence at the rate $\bigo(n^{-1/3+\kappa})$; \cite{BaoG24U} used the upper bounds of the ultra high order cumulants to derive the quantitative versions of the CLT for the trace of any given self-adjoint polynomial in generally distributed Wigner matrices; Most closely related to the present work, \cite{CuiH25R} established the convergence rate $n^{-1/2+\kappa}$ in the CLT for LSS of general sample covariance matrices.

\subsection{Organization}
	The remainder of the paper is organized as follows. Section \ref{Bernstein_sec} introduces the basic properties of Bernstein polynomial approximations used in our analysis. The proof of Theorem \ref{FunctionCLT} is shown in Section \ref{profFunCLT}, and Section \ref{ProfRateCLT} is devoted to the proof of Theorem \ref{RateFunctionCLT}. Section \ref{uselema} collects several technical lemmas used in throughout the paper. Appendix \ref{ProfQj} provides the proof of \eqref{Qjdecom} in Section \ref{profFunCLT}.

\subsection{Notation and conventions}
Throughout this paper, $n$ is considered the fundamental large parameter, and $K$ denotes a positive constant that does  not depend on $n$. Multiple occurrences of $K$, even within a single formula, do not necessarily represent the same value. For a parameter $a$, $K_a$ signifies a constant that may depend on $a$. The notation $A=\bigo(B)$ indicates $\lvert A\rvert\leq KB$, while $A_n=o(B_n)$ means $A_n/B_n \to 0$. $A_n=\bigo_p(B_n)$ means $|A_n|\leq KB_n$ in probability, $A_n=\bigo_{a.s.}(B_n)$ means $|A_n|\leq KB_n$ a.s.,  $A_n=o_p(B_n)$ means $A_n/B_n$ tending to $0$ in probability as $n\to\infty$. We use $\lVert \bbA\rVert$ for the spectral norm of a matrix $\bbA$, and $\lVert \bbu\rVert$ for the $L^2$-norm of a vector $\bbu$. To facilitate notation, we will let $\bbT=\bbT_p$. Under Assumption \ref{assumpopulationmatrix}, we may assume $\lVert \bbT\rVert\leq 1$ for all $p$. Let $\mathbb{E}_{0}(\cdot)$ be the expectation, $\mathbb{E}_{j}(\cdot)$ be the conditional expectation with respect to the $\sigma$-field generated by $\bbr_{1}, \dots, \bbr_{j}$. We use $\mathrm{tr}(\bbA), \bbA^{\top}$ and $\bbA^*$ to denote the trace, transpose and conjugate transpose of matrix $\bbA,$ respectively.

\subsection{Contribution and proof strategy}
We now summarize the main contribution of the paper and the blueprint of the proof. The principal contribution is the establishment of CLTs for the LSSs of general sample covariance matrices, together with quantitative convergence rates, under the test functions belong to $C^3$. In this sense, our work extends \cite{BaiW10F}, which treated the case $\bbT_p=\bbI_p$ with $f\in C^4$, and also relaxes the analyticity assumption imposed in \cite{CuiH25R}.

 Our argument begins with the Bernstein approximation of the test function $f$. Owing to Lemma \ref{convergence_rate_ESD} (originating from \cite{BaiH12C}), we are able to control the remainder terms. As a result, the problem reduces to the study of the approximating statistic $\int f_m(x)d G_n(x)$, where $f_m$ is shown in \eqref{fm}. 

\textit{$\bullet$ Functional CLTs for general sample covariance matrices.} Following the initial procedures, where the random variables are truncated at $\eta n^{1/4}$ and subsequently normalized, our next goal is to establish the connection between the centered LSS and the Stieltjes transform. By Cauchy's integral formula, the centered LSS can be rewritten as  
\begin{align*}
	\int f_m(x)d[G_n(x)-\mathbb{E}G_n(x)]=-\frac{p}{2\pi i}\oint_{\mathcal{C}}f_m(z)[s_{F^{\bbB_n}}(z)-\mathbb{E}s_{F^{\bbB_n}}(z)]dz,
\end{align*}
where the contour $\mathcal{C}$ is closed, oriented in the positive direction in the complex plane, and encloses the support set \eqref{supportset}. Compared with the Helffer-Sj\"ostrand formula employed in \cite{NajimY16G,DingW25G}, our presentation is more direct: although both methods rely on the Stieltjes transform, the latter ultimately leads to a double integral representation. By contrast, the contour integral formulation is used here is structurally simpler and is more convenient for calculations by using residue theorem and related techniques. Consistent with the approach based on martingale decomposition and integration by parts adopted in Section 2 of \cite{BaiS04C}, the centered Stieltjes transform can be represented as a sum of martingale difference sequences. Further details are given at the initial of Section \ref{secrandompart}.

We then apply the martingale CLT to derive the covariance function. Compared with \cite{BaiS04C,BaiW10F}, our approach differs in the following three aspects.
\begin{itemize}
	\item First, for the estimation of the Stieltjes transform, we provide bounds that are uniform in $z\in\calC$, rather than scaled by a factor depending on $\Im z$. Such estimates enable us to carry out the analysis along the entire contour $\calC$, rather than being restricted only to the horizontal segments of the contour; 
	\item Second, compared with the decomposition used in \cite{BaiS04C,BaiW10F}, we introduce formulas \eqref{construct an equation 1} and \eqref{construct an equation 2}, making the method conceptually more accessible;
	\item Third, after making several substitutions (replacing  $(y,H,\underline{s}^0)$ by $(y_n,H_p,\underline{s}_n^0)$), we identify the convergence rates of the corresponding mean and covariance terms as $\bigo(n^{-1})$ and $\bigo(n^{-1/2})$ respectively, seeing	\eqref{RGinvar}, \eqref{CGinvar}, \eqref{sup M_n^2 in RG case}, \eqref{sup M_n^2 in CG case}.
\end{itemize}

\textit{$\bullet$ Convergence rates of the functional CLT.} Building on the results of \cite{CuiH25R}, we first derive the convergence rates for the normalized approximating statistic $(\int f_m(x) dG_n(x)-\mu_n(f_m))/\sigma_n(f_m)$. By estimating the orders of the Bernstein remainder terms, $|\int f_m(x) dG_n(x)-\int f(x) dG_n(x)|$, $|\mu_n(f_m)-\mu_n(f)|$, and $|\sigma_n(f_m)-\sigma_n(f)|$, and using Lemma \ref{Chen}, we obtain the conclusion of Theorem \ref{RateFunctionCLT}.

 \section{Bernstein polynomial approximations}\label{Bernstein_sec} 
 The important tool of this paper is the Bernstein polynomials, which is 
 \begin{align*}
 	\tdf_m(y)=\sum_{k=0}^{m}{m \choose k} y^k (1-y)^{m-k}\tdf\left(\frac{k}{m}\right),
 \end{align*}
 where $\tdf(y)$ is a continuous function on the interval $\left[0,1\right]$. It is well known that $\tdf_m(y)\to \tdf(y)$ uniformly on $\left[0,1\right]$ as $m\to\infty$. A perspective on the probability of this polynomial is that
 \begin{align*}
 	\tdf_m(y)=\bE\left[ \tdf \left(\frac{X_m}{m}\right)\right],
 \end{align*}
 where $X_m$ follows a Bernoulli distribution $ \mbox{B}(m,y)$. 
 
 Suppose that $\tdf(y)\in C^3\left[0,1\right]$. A Taylor expansion gives 
 \begin{align*}
 	\tdf\left(\frac{k}{m}\right)= & \tdf(y)+\left(\frac{k}{m}-y\right)\tdf^{\prime}(y)+\frac{1}{2}\left(\frac{k}{m}-y\right)^2\tdf^{\prime\prime}(y) +\frac{1}{3!}\left(\frac{k}{m}-y\right)^3\tdf^{(3)}(\xi_y)
 \end{align*}
 where $\xi_y$ is a number between $k/m$ and $y$. Hence,
 \begin{align}\label{berstein_residus}
 	\tdf_m(f)-\tdf(y)=\frac{y(1-y)\tdf^{\prime\prime}(y)}{2m}+\bigo\left(\frac{1}{m^2}\right).
 \end{align}
 
 Let $\upsilon\in(0,1/2)$. For any $x\in\left[x_l,x_r\right]$, we perform a linear transformation $y=Lx+c$, where $L=(1-2\upsilon)/(x_r-x_l)$ and $c=((x_l+x_r)\upsilon-x_L)/(x_r-x_l)$, then $y\in[\upsilon,1-\upsilon]$. Define $\tdf(y)\triangleq f((y-c)/L)=f(x)$, $y\in[\upsilon,1-\upsilon]$ and
 \begin{align}\label{fm}
 	f_m(x)\triangleq \tdf_m(y)=\sum_{k=0}^{m}{m \choose k}y^k(1-y)^{m-k}\tdf\left(\frac{k}{m}\right).
 \end{align}
 From \eqref{berstein_residus}, we have
 \begin{align}\label{fm-f}
 	f_m(x)-f(x)=\tdf_m(y)-\tdf(y)=\frac{y(1-y)\tdf^{\prime\prime}(y)}{2m}+\bigo\left(\frac{1}{m^2}\right).
 \end{align}
 
 
 
 Let $\tdh(y)\triangleq y(1-y)\tdf^{\prime\prime}(y)$, similarly, $h_m(x)=\tdh_m(y)$. So LSS could be split into three parts: \begin{align}\label{aboutdelta}
 	\int f(x) dG_n(x) &=p\int f(x)\left[F^{\bbB_n}-F^{y_n,H_n}\right](d x)  \\
 	\nonumber &=p\int f_m(x) \left[F^{\bbB_n}-F^{y_n,H_n}\right](d x) \\
 	\nonumber &\quad -\frac{p}{2m}\int h_m(x) \left[F^{\bbB_n}-F^{y_n,H_n}\right](d x)  \\
 	\nonumber &\quad +p \int \left( f(x)-f_m(x)+\frac{1}{2m}h_m(x)\right)\left[F^{\bbB_n}-F^{y_n,H_n}\right](d x) \\
 	\nonumber 	&\triangleq \Delta_1+\Delta_2+\Delta_3.
 \end{align}

Let $v_0>0$. Clearly, the two polynomials $f_m(x)$ and $\tdf_m(y)$, defined only on the real line, can be extended to $\left[x_l,x_r\right]\times \left[-v_0,v_0\right]$ and $\left[\epsilon,1-\epsilon\right]\times \left[-Lv_0,Lv_0\right]$, respectively. Due that any analytic function is bounded on a compact domain, it follows that $f_m(z)$ and $h_m(z)$ are bounded on a bounded contour.
 
 \section{Proof of Theorem \ref{FunctionCLT}}\label{profFunCLT}
 
 \subsection{Truncation and normalization}\label{trunnorm}
 	 This subsection will be dedicated to the initialization of the random variables, including truncation and normalization. The following discussion will show that these procedures do not affect the establishment of the CLT.
	 
	 Under Assumption \ref{assum8th}, for any $\eta>0$, 
	 \begin{align*}
	 	\eta^{-8}\bE|x_{11}|^8I\{|x_{11}|\geq \eta n^{1/4}\}\to 0.
	 \end{align*}
	 According to Lemma 15 in \cite{LiB16C}, we can select a slowly decreasing sequence of constants $\eta_n\to 0$ such that
	 \begin{align}\label{truncation_condition}
	  	\eta_n^{-8}\bE|x_{11}|^8I\{|x_{11}|\geq \eta_n n^{1/4}\}\to 0.
	 \end{align}
	 Let $\widehat{x}_{ij}=x_{ij}I\{|x_{ij}|<\eta_n n^{1/4}\}$ and  $\widetilde{x}_{ij}=(\widehat{x}_{ij}-\mathbb{E}\widehat{x}_{ij})/\sigma_{ij}$, where $\sigma_{ij}^2=\mathbb{E}\lvert \widehat{x}_{ij}-\mathbb{E}\widehat{x}_{ij} \rvert^2$. We use $\widehat{\bbX}_n$ and $\widetilde{\bbX}_n$ to represent the random matrices with entries $\widehat{x}_{ij}$ and $\widetilde{x}_{ij}$, respectively. Let $\widehat{\bbB}_n=(1/n)\bbT^{1/2}\widehat{\bbX}_n\widehat{\bbX}_n^{*}\bbT^{1/2}$, $\widetilde{\bbB}_n = (1/n)\bbT^{1/2}\widetilde{\bbX}_n\widetilde{\bbX}_n^{*}\bbT^{1/2}$. Then, let $\widehat{G}_n(x)$ and $\widetilde{G}_n(x)$ denote the analogues of $G_n(x)$ with $\bbB_n$ replaced by $\widehat{\bbB}_n$ and $\widetilde{\bbB}_n$.
	 
	 We then have, due to \eqref{truncation_condition},
	 \begin{align}\label{trun}
	 	\bP\left(G_n\neq \widehat{G}_n\right)&=\bP\left(\bbB_n\neq\widehat{\bbB}_n\right)\leq np\bP\left(|x_{11}|\geq \eta_n n^{1/4}\right) \\
	 	\nonumber&\leq \eta_n^{-8}\bE\left[|x_{11}|^{8}I\{|x_{11}|\geq \eta_n n^{1/4}\}\right]=o(1),
	 \end{align}
	 \begin{align}\label{Ehatx}
	 	|\bE\widehat{x}_{ij}| &=\left| \bE x_{ij}I\{|x_{11}|\geq \eta_n n^{1/4}\}\right| \leq (\eta_n n^{1/4})^{-7}\bE\left[|x_{11}|^{8}I\{|x_{11}|\geq \eta_n n^{1/4}\}\right] \\
	 	\nonumber &=o(\eta_n n^{-7/4}).
	 \end{align}
	 Further,
	 \begin{align}\label{sig-1}
	 	\left(1-\sigma_{ij}^{-1}\right)^2&=\frac{(\mathbb{E}\lvert x_{ij}\rvert^2-\mathbb{E}\lvert\widehat{x}_{ij}-\mathbb{E}\widehat{x}_{ij}\rvert^2)^2}{\sigma_{ij}^2(1+\sigma_{ij})^2} \\
	 	\nonumber &\leq 2\frac{\left[ \mathbb{E}\lvert x_{ij}\rvert^2 I\left\lbrace \lvert x_{ij}\rvert \geq \eta_n n^{1/4}\right\rbrace \right]^2 + (\mathbb{E}\widehat{x}_{ij})^4 }{\sigma_{ij}^2(1+\sigma_{ij})^2}  \\
	 	\nonumber &\leq K\left[(\eta_nn^{1/4})^{-6}\bE\left[|x_{11}|^{8}I\{|x_{11}|\geq \eta_n n^{1/4}\}\right]\right]^2 +o(\eta_n^4 n^{-7}) \\
	 	\nonumber &=o(\eta_n^4 n^{-3}).
	 \end{align}
	 
	 Following the methodology and bounds from the proof of Lemma 2.7 in \cite{Bai99M}, we have, for each $j=1,\dots,k$,  
	 \begin{align*}
	 	&  \mathbb{E} \left| \int f_j(x)d\widehat{G}_n(x)-\int f_j(x)d\widetilde{G}_n(x)\right|  \\
	 	\nonumber =& \mathbb{E}\left| \sum_{k=1}^{p} [ f_j(\lambda_k^{\widehat{\bbB}_n})-f_j(\lambda_k^{\widetilde{\bbB}_n})] \right| \leq K \mathbb{E} \left| \sum_{k=1}^{p}(\lambda_k^{\hat{\bbB}_n}-\lambda_k^{\tilde{\bbB}_n}) \right|  \\
	 	\nonumber \leq& K \left[\sum_{k=1}^{p}\mathbb{E}\left(\sqrt{\lambda_k^{\widehat{\bbB}_n}}+\sqrt{\lambda_k^{\widetilde{\bbB}_n}}\right)^2\right]^{1/2}\left[\sum_{k=1}^{p} \mathbb{E}\bigg\lvert \sqrt{\lambda_k^{\hat{\bbB}_n}}-\sqrt{\lambda_k^{\tilde{\bbB}_n}}\bigg\rvert^2\right]^{1/2}  \\
	 	\nonumber  \leq& K\left[n^{-1}\mathbb{E} \operatorname{tr}\bbT^{1/2}(\widehat{\bbX}_n-\widetilde{\bbX}_n)(\widehat{\bbX}_n-\widetilde{\bbX}_n)^{*}\bbT^{1/2}\right]^{1/2}\\
	 	\nonumber &\quad\quad \left[2n^{-1}\mathbb{E} \operatorname{tr} \bbT^{1/2}(\widehat{\bbX}_n\widehat{\bbX}_n^{*})\bbT^{1/2}\right]^{1/2}.
	 \end{align*}
	 Using the above formulas \eqref{Ehatx} and \eqref{sig-1}, we can separately estimate the following two terms:
	 	\begin{align*}
	 	& n^{-1}\mathbb{E} \operatorname{tr}\bbT^{1/2}(\widehat{\bbX}_n-\widetilde{\bbX}_n)(\widehat{\bbX}_n-\widetilde{\bbX}_n)^{*}\bbT^{1/2}
	 	\leq n^{-1} \lVert \bbT\rVert\sum_{i,j} \mathbb{E}\bigg\lvert (1-\sigma_{ij}^{-1})\widehat{x}_{ij}+\frac{\mathbb{E} \widehat{x}_{ij}}{\sigma_{ij}}\bigg\rvert^2 \\
	 	\leq& 2n^{-1}\sum_{i,j}\left[ (1-\sigma_{ij}^{-1})^2\mathbb{E}\lvert \widehat{x}_{ij}\rvert^2+\frac{1}{\sigma_{ij}^2} \mathbb{E}\lvert \widehat{x}_{ij}\rvert^2 \right]  
	 	\leq 2n^{-1} np \left[ (1-\sigma_{11}^{-1})^2+\sigma_{11}^{-2}\lvert\mathbb{E} \widehat{x}_{11}\rvert^2 \right]  \\
	 	= &o(\eta_n^4 n^{-2}).
	 \end{align*}
	 and
	 \begin{align*}
	 	&  n^{-1}\mathbb{E} \operatorname{tr} \bbT^{1/2}(\widehat{\bbX}_n\widehat{\bbX}_n^{*})\bbT^{1/2} 
	 	\leq n^{-1}\lVert \bbT \rVert \sum_{i,j} \mathbb{E}\left[\lvert \widehat{x}_{ij}\rvert^2 +\lvert \widetilde{x}_{ij} \rvert^2 \right]  \\
	 	\leq& K n^{-1} np \left[ \mathbb{E}\lvert x_{11}\rvert^2+\frac{\mathbb{E}\lvert \widehat{x}_{11}-\mathbb{E}\widehat{x}_{11}\rvert^2}{\sigma_{11}} \right]  
	 	\leq K p\left[ \mathbb{E}\lvert x_{11}\rvert^2 + \frac{2\mathbb{E}\lvert \widehat{x}_{11}\rvert^2}{\sigma_{11}}\right] 
	 	= \bigo(p). 
	 \end{align*}
	 Therefore,
	 \begin{align}\label{norm}
	 	\mathbb{E}\left| \int f_j(x)d\widehat{G}_n(x)-\int f_j(x)d\widetilde{G}_n(x)\right|=o(\eta_n^{2}n^{-1/2}).
	 \end{align}
	 
	 From \eqref{trun} and \eqref{norm}, we obtain
	 \begin{align*}
	 	 \int f_j(x) d G_n(x)=\int f_j(x) d\widetilde{G}_n(x)+o_p(1),
	 \end{align*}
	 where $o_p(1)\stackrel{p}{\rightarrow} 0$ as $n\to\infty$. Henceforth, we assume the underlying variables are truncated at $\eta_n n^{1/4}$, centralized and normalized. For simplicity, we will suppress all sub- or superscripts on these variables and assume $\lvert x_{ij}\rvert<\eta_n n^{1/4}$, $\mathbb{E}x_{ij}=0$, $\mathbb{E}\lvert x_{ij}\rvert^2=1$, $\mathbb{E}\lvert x_{ij}\rvert^{8}<\infty$. For  Assumption \ref{assumCG}, $\mathbb{E} x_{ij}^2=o(n^{-3/2}\eta_n^2)$.
	 From  \cite{BaiY93L} and \cite{YinB88L}, after truncation and normalization,  
	 \begin{align}\label{outofsupport}
	 	\mathbb{P}(\lVert \bbB_n\rVert\geq \mu_1)=o(n^{-l})\quad \quad and \quad\quad \mathbb{P}(\lambda_{min}^{\bbB_n}\leq \mu_2)=o(n^{-l})
	 \end{align}
	 hold for any $\mu_1>\limsup_p\lVert \bbT_p\rVert(1+\sqrt{y})^2$,  $0<\mu_2<\liminf_p\lambda_{min}^{\bbT_p}I_{(0,1)}(y)(1-\sqrt{y})^2$ and $l>0$.

	 \subsection{Stieltjes transformation and LSS}\label{zanting}
 	In this subsection, we will introduce the relationship between  the LSS and the Stieltjes transformation. Define
	\begin{align}\label{M_n}
		M_n(z)=p[s_n(z)-s_n^0(z)]=n[\underline{s}_n(z)-\underline{s}_n^0(z)],
	\end{align}
	where $s_n(z):=s_{F^{\bbB_n}}(z)$ denotes the Stieltjes transform of $F^{\bbB_n}$, $s_n^0(z)$ denotes the Stieltjes transform of $F^{y_n,H_p}$, $\underline{s}_n(z)$ and $\underline{s}_n^0(z)$ denote the Stieltjes transform of $\underline{F}^{\bbB_n}$ and $\underline{F}^{y_n,H_p}$.

	By using the Cauchy's integral formula, we rewrite the LSS into 
	\begin{align*}
		\int f(x)dG_n(x)=-\frac{1}{2\pi i}\oint_{\mathcal{C}}f(z)M_n(z)dz,
	\end{align*}
	where the contour $\mathcal{C}$ is closed and taken in the positive direction in the complex plane containing the support of $G_n$. In what follows, we specify the selection of a suitable contour $\mathcal{C}$.

	Let $v_0>0$ and $\epsilon>0$ be arbitrary. Let $x_r=\lambda_{max}^{\bbT}(1+\sqrt{y})^2+\epsilon$. Let $x_l$ be any negative number if the left point of \eqref{supportset} is zero. Otherwise, choose $x_l=\lambda_{min}^{\bbT}I_{(0,1)}(y)(1-\sqrt{y})^2-\epsilon$. Define
	$$\mathcal{C}_u=\left\lbrace x+iv_0:x\in[x_l,x_r]\right\rbrace, 
	~~\mathcal{C}_l=\left\lbrace x_l+iv:v\in[0,v_0]\right\rbrace,$$
	$$\mathcal{C}_r= \left\lbrace x_r+iv:v\in[0,v_0]\right\rbrace,~~\mathcal{C}^{+}\equiv \mathcal{C}_l \cup \mathcal{C}_u \cup \mathcal{C}_r. $$
	Then  $\mathcal{C}=\mathcal{C}^{+}\cup \overline{\mathcal{C}^{+}}$, seeing Fig \ref{contour 1}.
	
	\begin{figure}[htbp]
		\centering
		\begin{tikzpicture}[>=latex]
			\draw[->] (-1,0)--(7,0)node[below]{$u$};
			\draw[->] (0,-2)--(0,2)node[left]{$v$};
			\node at (-.3, -.3) {$O$};
			\draw (0,1)node[left]{$v_0$}--(.1,1);
			\draw (0,-1)node[left]{$-v_0$}--(.1,-1);
			
			\node at (2,-.3)[left] {$x_l$};
			\draw (3,0)--(3,.1);
			\draw (5,0)--(5,.1);
			\draw (3,0)--node[below]{$\eqref{supportset}$}(5,0);
			\node at (6,-.3) [right] {$x_r$};
			\node at (6.1,.5) [right] {$\mathcal{C}$};

			\begin{scope}[color=black, thick, every node/.style={sloped,allow upside down}]
				\draw (2,1)-- node {\midarrow} (2,-1);
				\draw (2,-1)-- node {\midarrow} (6,-1);
				\draw (6,-1)-- node {\midarrow} (6,1);
				\draw (6,1)-- node {\midarrow} (2,1);
			\end{scope}

		\end{tikzpicture}
		\caption{Contour $\mathcal{C}$.}
		\label{contour 1}
	\end{figure}

	Let
	\begin{align}\label{Xi_n}
		\Xi_n=\left\lbrace \lambda_{min}^{B_n}\leq x_l+\epsilon/2\quad \mbox{or}\quad\lambda_{max}^{B_n}\geq x_r-\epsilon/2\right\rbrace.
	\end{align}
	Due to (\ref{outofsupport}), for any $l>0$, we have 
	\begin{align}\label{order of Xi_n}
		\mathbb{P}(\Xi_n)=o(n^{-l}).
	\end{align}
	Since the support of $F^{y_n,H_p}$ is contained within $[x_l-\epsilon/2,x_r+\epsilon/2]$, it follows from the CLT for $	\int f(x) dG_n(x)$ established by \cite{BaiS04C}, 
	\begin{align*}
		\int f(x) dG_n(x)  \stackrel{n^{-\frac{l}{4}}}{\sim} I(\Xi_n^{c})\int f(x)dG_n(x),
	\end{align*}
	where $\stackrel{\epsilon_n}{\sim}$ defined in Remark \ref{K-S_dist_equi}. For simplicity, we omit the notation  $I(\Xi_n^{c})$ in the sequel.

In order to use the Cauchy's integral formula, we need to replace $f$ with $f_m$ defined in \eqref{fm}. Therefore, it's necessary to estimate the terms $\Delta_2$ and $\Delta_3$ in \eqref{aboutdelta}. Taking $m=[n^{3/5+\epsilon_0}]$ for some $\epsilon_0>0$. According to Lemma \ref{convergence_rate_ESD},
$$\sup_{x}|F^{\bbB_n}(x)-F^{y_n,H_n}(x)|=\bigo_{p}(n^{-2/5}),$$
combining $h^{\prime}(x)=\bigo(m^{-1})$ and $\left(f(x)-f_m(x)+\frac{1}{2m}h_m(x)\right)^{\prime}=\bigo(m^{-2})$, 
\begin{align*}
	& \Delta_2=\bigo_p(n^{-\epsilon_0})=o_p(1), \\
	& \Delta_3=\bigo_p(n^{-3/5-2\epsilon_0})=o_p(1).	
\end{align*}
Thus, due that $f_m$ defined in \eqref{fm} is an analytic function,
\begin{align*}
	\int f(x)dG_n(x)=-\frac{1}{2\pi i}\oint_{\mathcal{C}}f_m(z)M_n(z)dz+o_p(1).	
\end{align*}

For $z\in\mathcal{C}$, we decompose $M_n(z)=M_n^1(z)+M_n^2(z)$, where
\begin{align*}
	M_n^1(z)=p\left[s_n(z)-\bE s_n(z)\right],
\end{align*}
and
\begin{align*}
	M_n^2(z)=p\left[\bE s_n(z)-s_n^0(z)\right].
\end{align*} 
LSSs are divide into two parts: the random part
\begin{align}\label{the random part}
	-\frac{1}{2\pi i}\oint_{\calC} f_m(z)M_n^1(z)dz,
\end{align}
and non-random part
\begin{align}\label{the nonrandom part}
	-\frac{1}{2\pi i}\oint_{\calC} f_m(z)M_n^2(z)dz.
\end{align}
 The subsequent analysis investigate the two components,  \eqref{the random part} and \eqref{the nonrandom part}, respectively.

 \subsection{Random part}\label{secrandompart}
  	After the above initial steps, we firstly introduce some notations adapted from \cite{BaiS04C}, which is used frequently. Let $\bbr_{j}=(1 / \sqrt{n}) \bbT^{1/2} \bbX_{\cdot j}$, where $\bbX_{\cdot j}$ denotes the $j$-th column of $\bbX_n$. Denote $\bbD(z)=\bbB_{n}-z\bbI$, $ \bbD_{j}(z)=\bbD(z)-\bbr_{j}\bbr_{j}^{*}$, $\bbD_{ij}(z)=\bbD(z)-\bbr_i\bbr_i^{*}-\bbr_j\bbr_j^{*}$, and
 \begin{align*}
 	\varepsilon_{j}(z)=&\bbr_{j}^{*} \bbD_{j}^{-1}(z) \bbr_{j}-\frac{1}{n} \operatorname{tr} \bbT \bbD_{j}^{-1}(z), ~~
 	\gamma_j(z)=\bbr_j^*\bbD_j^{-1}(z)\bbr_j-\frac{1}{n}\mathbb{E}\operatorname{tr} \bbT\bbD_j^{-1}(z),\\
 	\beta_{j}(z)=&\frac{1}{1+\bbr_{j}^{*} \bbD_{j}^{-1}(z) \bbr_{j}}, ~~ \widetilde{\beta}_{j}(z)=\frac{1}{1+n^{-1} \operatorname{tr} \bbT\bbD_{j}^{-1}(z)}, ~~ b_{j}(z)=\frac{1}{1+n^{-1} \mathbb{E} \operatorname{tr} \bbT \bbD_{j}^{-1}(z)}, \\
 	\beta_{ij}(z)=&\frac{1}{1+\bbr_{i}^{*} \bbD_{ij}^{-1}(z) \bbr_{i}}, ~~ \widetilde{\beta}_{ij}(z)=\frac{1}{1+n^{-1}\operatorname{tr}\bbT\bbD_{ij}^{-1}(z)}, ~~ b(z)=\frac{1}{1+n^{-1} \mathbb{E} \operatorname{tr} \bbT \bbD^{-1}(z)},\\
 	\epsilon_{ij}(z)=&\bbr_i^*\bbD_{ij}^{-1}(z)\bbr_i-\frac{1}{n}\operatorname{tr}\bbT\bbD_{ij}^{-1}(z), ~~ b_{ij}(z)=\frac{1}{1+n^{-1} \mathbb{E} \operatorname{tr} \bbT \bbD_{ij}^{-1}(z)}, \\
 	\gamma_{ij}(z)=&\bbr_i^*\bbD_{ij}^{-1}(z)\bbr_i-\frac{1}{n} \bE \operatorname{tr}\bbT\bbD_{ij}^{-1}(z).
 \end{align*}

 Now we are in the position to consider the random part \eqref{the random part}. Employing martingale decomposition, integration by parts, we have that 
 \begin{align}\label{decomposition of random part}
 	&  p\int f_m(x)d[F^{\bbB_n}(x)-\mathbb{E}F^{\bbB_n}(x)] 
 	=-\frac{p}{2\pi i}\oint_{\mathcal{C}}f_m(z)[s_{n}(z)-\mathbb{E}s_{n }(z)]dz  \\
 	\nonumber 
 	=& -\frac{1}{2\pi i}\sum_{j=1}^{n}\oint_{\mathcal{C}} f_m^{\prime}(z) (\mathbb{E}_j-\mathbb{E}_{j-1})\left[\varepsilon_j(z)b_j(z)+Q_j(z)\right]dz\\=& -\frac{1}{2\pi i}\sum_{j=1}^{n}\oint_{\mathcal{C}} f_m^{\prime}(z) (\mathbb{E}_j-\mathbb{E}_{j-1})\varepsilon_j(z)b_j(z)dz+o_p(1)\triangleq-\frac{1}{2\pi i}\sum_{j=1}^{n}Y_j+o_p(1), \label{Qjdecom}
 \end{align}
 where $Q_j(z)=R_j(z)+\varepsilon_j(z)(\widetilde{\beta}_j(z)-b_j(z))$, $R_j(z)=\Log (1+\varepsilon_j(z)\widetilde{\beta}_j(z))-\varepsilon_j(z)\widetilde{\beta}_j(z)$. Throughout the paper, the logarithm is defined using the principal branch. 
More precisely, for $z \in \mathbb{C} \setminus (-\infty,0]$, we define
$
\Log z = \ln |z| + i \text{arg} z,
$
where the argument satisfies
$
\text{arg} z \in (-\pi,\pi).
$
Note that the event $\{|\epsilon_j(z)\tilde{\beta}_j(z)|< 1/2\}$ holds with high probability (seeing in Appendix \ref{ProfQj}). Since the proof of \eqref{Qjdecom} is routine in random matrix theory, we postpone it to Appendix \ref{ProfQj}.
 Since $\sum_{j=1}^{n}Y_j$ is a sum of martingale difference sequence, to apply the martingale CLT (Theorem 35.12 of \cite{Billingsley95}, see Lemma \ref{yangclt}) to it, we need to check two conditions, the Lyapunov condition 
$$  \sum_{j=1}^{n}\mathbb{E}\left|Y_j\right|^4 \rightarrow0, $$ and conditional covariance condition. Firstly, we check the Lyapunov condition.
\begin{proof}[Proof of Lyapunov condition]
	Since 
	\begin{align*}
		 \sum_{j=1}^{n}\mathbb{E}\left|Y_j\right|^4=&\sum_{j=1}^{n}\mathbb{E}\left| \oint_\mathcal{C} f_m^{\prime}(z) (\mathbb{E}_j-\mathbb{E}_{j-1})\varepsilon_j(z)b_j(z)dz \right|^4\\
		 \leq& K \sum_{j=1}^{n}\oint_\mathcal{C}\mathbb{E}\left|\varepsilon_j(z)b_j(z) \right|^4dz \leq K \sum_{j=1}^{n}\oint_\mathcal{C}(\mathbb{E}\left|b_j(z) \right|^{12})^{1/3}(\mathbb{E}\left|\varepsilon_j(z) \right|^{6})^{2/3}dz
	\end{align*}
and from Lemma \ref{QMTr}, we have that $\mathbb{E}\left|\varepsilon_j \right|^6=\bigo(n^{-3}), $ then $\sum_{j=1}^{n}\mathbb{E}\left|Y_j\right|^4=\bigo(n^{-1})$. The proof of 	Lyapunov condition is finished.
\end{proof}
Secondly, we check the second condition.
\begin{proof}[Proof of conditional covariance]
Since
	\begin{align}\label{randpartterm}
		&-\frac{1}{4\pi^2}\sum_{j=1}^{n}\mathbb{E}_{j-1}\left[Y_j(f_m)Y_j(g_m)\right]\\
		\nonumber=&-\frac{1}{4\pi^2}\sum_{j=1}^{n}\mathbb{E}_{j-1}[\oint_{\mathcal{C}_1} f_m^{\prime}(z_1) (\mathbb{E}_j-\mathbb{E}_{j-1})\varepsilon_j(z_1)b_j(z_1)dz_1\oint_{\mathcal{C}_2} g_m^{\prime}(z_2) (\mathbb{E}_j-\mathbb{E}_{j-1})\varepsilon_j(z_2)b_j(z_2)dz_2]\\
		\nonumber=&-\frac{1}{4\pi^2}\oint_{\mathcal{C}_1}\oint_{\mathcal{C}_2} f_m^{\prime}(z_1) g_m^{\prime}(z_2)\sum_{j=1}^{n} \mathbb{E}_{j-1}[  \mathbb{E}_j\varepsilon_j(z_1)b_j(z_1)  \mathbb{E}_j\varepsilon_j(z_2)b_j(z_2)]dz_2dz_1,
	\end{align}
then we show that 
$ \eqref{randpartterm}\rightarrow \Cov(X_f,X_g) ~\text{in probability}, $	
where $\Cov(X_f,X_g)$ is defined in Theorem \ref{FunctionCLT}.

From Lemma \ref{E|trTDj-1-EtrTDj-1|},  we have
\begin{align}\label{|b_j-Ebeta_j|}
	&\lvert b_j(z_i)-\mathbb{E}\beta_{j}(z_i)\rvert =\lvert \mathbb{E}(b_j(z_i)-\beta_{j}(z_i))\rvert  \\
	\nonumber\quad =&\lvert b_j(z_i)\mathbb{E}[ \beta_{j}(z_i)(\bbr_j^{*}\bbD_j^{-1}(z_i)\bbr_j-n^{-1}\mathbb{E}\operatorname{tr}\bbT\bbD_j^{-1}(z_i))]\rvert \\
	\nonumber \leq& \lvert b_j(z_i)\mathbb{E}[\widetilde{\beta}_{j}(z_i)(\bbr_j^{*}\bbD_j^{-1}(z_i)\bbr_j-n^{-1}\mathbb{E}\operatorname{tr}\bbT\bbD_j^{-1}(z_i))]\rvert \\
	\nonumber & +\lvert b_j(z_i)\mathbb{E}[\beta_{j}(z_i)\widetilde{\beta}_j(z_i)\varepsilon_j(z_i)(\bbr_j^{*}\bbD_j^{-1}(z_i)\bbr_j-n^{-1}\mathbb{E}\operatorname{tr}\bbT\bbD_j^{-1}(z_i))]\rvert 
	\leq Kn^{-1},
\end{align}
and
\begin{align}\label{|b-Ebeta_j|}
	&\lvert b(z_i)-\mathbb{E}\beta_j(z_i)\rvert=\lvert b(z_i)-b_j(z_i)+b_j(z_i)-\mathbb{E}\beta_j(z_i)\rvert \\
	\nonumber \leq& \lvert b(z_i)-b_j(z_i)\rvert+\lvert b_j(z_i)-\mathbb{E}\beta_j(z_i)\rvert \leq Kn^{-1}.
\end{align}
From the formula 
$
	\underline{s}_n(z_i)=-\frac{1}{z_in}\sum_{j=1}^{n}\beta_{j}(z_i)
$
(see (2.2) in \cite{Silverstein95S}), we have $(1/n)\sum_{j=1}^{n}\mathbb{E}\beta_j(z_i)=-z_i\mathbb{E}\underline{s}_n(z_i)$.
Due to \eqref{|b-Ebeta_j|} and Lemma  \ref{|Es_n-s_n^0|}, we get
\begin{align}\label{b(z)+zs_n^0(z)}
	& \lvert b(z_i)+z_i\underline{s}_n^0(z_i)\rvert=\lvert b(z_i)+z_i\mathbb{E}\underline{s}_n(z_i)+z_i\underline{s}_n^0(z_i)-z_i\mathbb{E}\underline{s}_n(z_i)\rvert \\
	\nonumber\leq& \lvert b(z_i)+z_i\mathbb{E}\underline{s}_n(z_i)\rvert+\lvert z_i\mathbb{E}\underline{s}_n(z_i)-z_i\underline{s}_n^0(z_i)\rvert \\
	\nonumber\leq& \lvert\frac{1}{n}\sum_{j=1}^{n}(b(z_i)-\mathbb{E}\beta_j(z_i))\rvert+Kn^{-1}\leq Kn^{-1}.
\end{align}
Due to \eqref{b(z)+zs_n^0(z)} and Lemmas \ref{E|beta_j-b_j|} and \ref{|b_j-b|}, we conclude that 
\begin{align}\label{b_j(z)+zs_n^0(z)}
	& \lvert b_j(z_i)+z_i\underline{s}_n^0(z_i)\rvert=\lvert b_j(z_i)-b(z_i)+b(z_i)+z_i\underline{s}_n^0(z_i)\rvert \\ 
	\nonumber\leq& \lvert b_j(z_i)-b(z_i)\vert + \lvert b(z_i)+z_i\underline{s}_n^0(z_i)\rvert \leq Kn^{-1},
\end{align}
and
\begin{align}\label{E|tilde beta_j(z)+zs_n^0(z)|}
	& \mathbb{E}\lvert\widetilde{\beta}_j(z_i)+z_i\underline{s}_n^0(z_i)\rvert=\mathbb{E}\lvert\widetilde{\beta}_j(z_i)-b_j(z_i)+b_j(z_i)+z_i\underline{s}_n^0(z_i)\rvert \\
	\nonumber\leq& \mathbb{E}\lvert\widetilde{\beta}_j(z_i)-b_j(z_i)\rvert+\lvert b_j(z_i)+z_i\underline{s}_n^0(z_i)\rvert \leq Kn^{-1}. 
\end{align}
Using \eqref{b_j(z)+zs_n^0(z)} and Lemma \ref{Q_QMTr}, we have
\begin{align*}
	&\lvert (b_j(z_1)+z_1\underline{s}_n^0(z_1))(b_j(z_2)+z_2\underline{s}_n^0(z_2))\mathbb{E}[\mathbb{E}_j(\varepsilon_j(z_1))\mathbb{E}_j(\varepsilon_j(z_2))]\rvert \\
	\nonumber\leq& Kn^{-2} (\mathbb{E}\lvert \mathbb{E}_j(\varepsilon_j(z_1))\rvert^2)^{1/2}(\mathbb{E}\lvert \mathbb{E}_j(\varepsilon_j(z_1))\rvert^2)^{1/2}\leq Kn^{-3}.
\end{align*}
Similarly,
\begin{align*}
	\lvert (b_j(z_1)+z_1\underline{s}_n^0(z_1))(-z_2\underline{s}_n^0(z_2))\mathbb{E}(\mathbb{E}_j\varepsilon_j(z_1)\mathbb{E}_j(\varepsilon_j(z_2)))\rvert\leq Kn^{-2}.
\end{align*}
This implies
\begin{align}\label{b_j(z_1)b_j(z_2)-z_1z_2s_n^0(z_1)s_n^0(z_2)}
	&\mathbb{E}\left| \sum_{j=1}^{n}b_j(z_1)b_j(z_2)\mathbb{E}_{j-1}[\mathbb{E}_j(\varepsilon_j(z_1))\mathbb{E}_j(\varepsilon_j(z_2))] \right. \\
	\nonumber &~~~~~~~~ \left. -z_1\underline{s}_n^0(z_1)z_2\underline{s}_n^0(z_2)\sum_{j=1}^{n}\mathbb{E}_{j-1}[\mathbb{E}_j(\varepsilon_j(z_1))\mathbb{E}_j(\varepsilon_j(z_2))]\right| \\
	\nonumber&=\bigo(n^{-1}).
\end{align}
Therefore, to obtain the limit of $\Gamma_n(z_1,z_2)$, our goal changes to 
\begin{align}\label{Ejejz1Ejejz2}
	z_1z_2\underline{s}_n^0(z_1)\underline{s}_n^0(z_2)\sum_{j=1}^{n}\mathbb{E}_{j-1}[ \mathbb{E}_j(\varepsilon_j(z_1))\mathbb{E}_j(\varepsilon_j(z_2))].
\end{align}	
	Due to Lemma \ref{product of two quadratic minus trace of two matrices}, we have
\begin{align*}
	&  \sum_{j=1}^{n}z_1z_2\underline{s}_n^0(z_1)\underline{s}_n^0(z_2)\mathbb{E}_{j-1} [\mathbb{E}_j(\varepsilon_j(z_1))\mathbb{E}_j(\varepsilon_j(z_2))] \\
	\nonumber=& \sum_{j=1}^{n}z_1z_2\underline{s}_n^0(z_1)\underline{s}_n^0(z_2)\mathbb{E}_{j-1}[(\bbr_j^{*}\mathbb{E}_j(\bbD_j^{-1}(z_1))\bbr_j-n^{-1}\operatorname{tr}\bbT\mathbb{E}_j(\bbD_j^{-1}(z_1)))  \\
	\nonumber&\quad\quad\quad\quad\quad\quad\quad\quad\quad\quad\quad (\bbr_j^{*}\mathbb{E}_j(\bbD_j^{-1}(z_2))\bbr_j-n^{-1}\operatorname{tr}\bbT\mathbb{E}_j(\bbD_j^{-1}(z_2)))] \\
	\nonumber=& \frac{\sum_{j=1}^{n}z_1z_2\underline{s}_n^0(z_1)\underline{s}_n^0(z_2)}{n^2}\left[  \beta_x\operatorname{tr}[ (\bbT\mathbb{E}_j\bbD_j^{-1}(z_1))\circ(\bbT\mathbb{E}_j\bbD_j^{-1}(z_2))]\right. \\
	\nonumber& ~ \left. +\alpha_x \operatorname{tr} \bbT\mathbb{E}_j(\bbD_j^{-1}(z_1))(\bbT\mathbb{E}_j(\bbD_j^{-1}(z_2)))^{\top}+ \operatorname{tr} \bbT\mathbb{E}_j(\bbD_j^{-1}(z_1))\bbT\mathbb{E}_j(\bbD_j^{-1}(z_2))  \right] \\
	\nonumber &~+\bigo(n^{-1}).
\end{align*}
Under the Assumption \ref{assumRG}, we have
\begin{align*}
	&  \sum_{j=1}^{n}z_1z_2\underline{s}_n^0(z_1)\underline{s}_n^0(z_2)\mathbb{E}_{j-1}\left[\mathbb{E}_j(\varepsilon_j(z_1))\mathbb{E}_j(\varepsilon_j(z_2))\right] \\
	=& \beta_x n^{-2}\sum_{j=1}^{n}z_1z_2\underline{s}_n^0(z_1)\underline{s}_n^0(z_2) \operatorname{tr}[ (\bbT\mathbb{E}_j\bbD_j^{-1}(z_1))\circ(\bbT\mathbb{E}_j\bbD_j^{-1}(z_2))] \\
	&+   2n^{-2}\sum_{j=1}^{n}z_1z_2\underline{s}_n^0(z_1)\underline{s}_n^0(z_2) \operatorname{tr} \bbT\mathbb{E}_j(\bbD_j^{-1}(z_1))\bbT\mathbb{E}_j(\bbD_j^{-1}(z_2)).
\end{align*}
Under the Assumption \ref{assumCG}, we have
\begin{align*}
	&\nonumber \sum_{j=1}^{n}z_1z_2\underline{s}_n^0(z_1)\underline{s}_n^0(z_2)\mathbb{E}_{j-1}\left[\mathbb{E}_j(\varepsilon_j(z_1))\mathbb{E}_j(\varepsilon_j(z_2))\right] \\
	=& \beta_x n^{-2}\sum_{j=1}^{n}z_1z_2\underline{s}_n^0(z_1)\underline{s}_n^0(z_2) \operatorname{tr}[ (\bbT\mathbb{E}_j\bbD_j^{-1}(z_1))\circ(\bbT\mathbb{E}_j\bbD_j^{-1}(z_2))]\\
	&+ n^{-2}\sum_{j=1}^{n}z_1z_2\underline{s}_n^0(z_1)\underline{s}_n^0(z_2) \operatorname{tr} \bbT\mathbb{E}_j(\bbD_j^{-1}(z_1))\bbT\mathbb{E}_j(\bbD_j^{-1}(z_2)).
\end{align*}	
In the following, we only need to consider the limits of
\begin{align}\label{hadmardprod}
	n^{-2}\sum_{j=1}^{n}z_1z_2\underline{s}_n^0(z_1)\underline{s}_n^0(z_2) \operatorname{tr}[ (\bbT\mathbb{E}_j\bbD_j^{-1}(z_1))\circ(\bbT\mathbb{E}_j\bbD_j^{-1}(z_2))]
\end{align}
and
\begin{align}\label{EtrTE_jD_jTE_jD_j}
	n^{-2}\sum_{j=1}^{n}z_1z_2\underline{s}_n^0(z_1)\underline{s}_n^0(z_2) \operatorname{tr} \bbT\mathbb{E}_j(\bbD_j^{-1}(z_1))\bbT\mathbb{E}_j(\bbD_j^{-1}(z_2)).
\end{align}

Inspired by \cite{PanZ08C}, we first consider the limit of \eqref{hadmardprod}. Let $\bbe_i$ be the $1\times p$ vector whose $i$-th entry is 1 and all other entries are 0. Due to Lemmas \ref{Burkholder_1}, \ref{bound moment of b_j and beta_j}, \ref{QMTr_1}, \eqref{D_j^{-1}-D_kj^{-1}},
\begin{align*}
		&\bE\left|\operatorname{tr}[ (\bbT(\mathbb{E}_j\bbD_j^{-1}(z_1)-\bE\bbD_j^{-1}(z_1)))\circ(\bbT\mathbb{E}_j\bbD_j^{-1}(z_2))\right| \\
		\nonumber =&\bE\left|\sum_{i=1}^{p}\sum_{s=1}^j \bbe_i^*\bbT (\mathbb{E}_s-\bE_{s-1})\bbD_{js}^{-1}(z_1)\bbr_s\bbr_s^*\bbD_{js}^{-1}(z_1)\bbe_i\bbe_i^*\bbT\bE_j\bbD_j^{-1}(z_2)\bbe_i\right| \\
		\nonumber \leq& \sum_{i=1}^p \left(\sum_{s=1}^j \bE\left|(\mathbb{E}_s-\bE_{s-1})(\bbr_s^*\bbD_{js}^{-1}\bbe_i\bbe_i^*\bbT\bbD_{js}^{-1}\bbr_s-n^{-1}\operatorname{tr}\bbT\bbD_{js}^{-1}\bbe_i\bbe_i^*\bbT\bbD_{js}^{-1}))\right|^2\right)^{1/2} \\
		\nonumber &\quad\quad\quad \left(\bE\left|\bbe_i^*\bbT\bE_j\bbD_j^{-1}(z_2)\bbe_i\right|^2\right)^{1/2} \\
		\nonumber = &\bigo(n^{1/2}).
\end{align*}
Thus,
\begin{align}\label{EjDj-EDj}
	& \bE\left| n^{-2}\sum_{j=1}^{n}z_1z_2\underline{s}_n^0(z_1)\underline{s}_n^0(z_2) \operatorname{tr}[ (\bbT\mathbb{E}_j\bbD_j^{-1}(z_1))\circ(\bbT\mathbb{E}_j\bbD_j^{-1}(z_2))]\right.\\
	\nonumber &\quad \left. -n^{-2}\sum_{j=1}^{n}z_1z_2\underline{s}_n^0(z_1)\underline{s}_n^0(z_2) \operatorname{tr}[ (\bbT\mathbb{E}\bbD_j^{-1}(z_1))\circ(\bbT\mathbb{E}\bbD_j^{-1}(z_2))]\right|=\bigo(n^{-1/2}).
\end{align}
It remains to find the limit of 
\begin{align}
	\sum_{i=1}^{p}\bbe_i^* \bbT\mathbb{E}\bbD_1^{-1}(z_1)\bbe_i \bbe_i^*\bbT\mathbb{E}\bbD_1^{-1}(z_2)\bbe_i.
\end{align}

Define 
$
	\widetilde{\bbT}_n(z)=z\bE \underline{s}_n(z)\bbT+z\bbI,
$
and
$$ \widetilde{\epsilon}_{j}(z)=\bbr_j^*\bbD_{1j}^{-1}(z)\bbe_i\bbe_i^*\bbT(\widetilde{\bbT}_n(z))^{-1}\bbr_j-n^{-1} \bbe_i^*\bbT(\widetilde{\bbT}_n(z))^{-1}\bbT\bbD_{1j}^{-1}(z)\bbe_i.$$
Due to \eqref{beta_1}, \eqref{beta_ij decomposition}, Lemmas \ref{a(v)_quadratic_minus_trace}, \ref{quadratic form minus trace of qudratic form}, 
$$\bE\underline{s}_n(z)=-\frac{1}{zn}\sum_{j=1}^{n}\bE\beta_{j}(z)=-\frac{1}{zn}\sum_{j=2}^{n}\bE\beta_{j1}(z)+\bigo(n^{-1}).$$
From \eqref{r_i^*(C+r_ir_i^*)^-1}, 
\begin{align}\label{eiTED-Tnei}
	&\bbe_i^*\bbT(\bE\bbD_1^{-1}(z)-(-\widetilde{\bbT}_n(z))^{-1})\bbe_i \\	
	\nonumber =&\sum_{j=2}^{n}\bE\left[ \beta_{j1}\bbe_i^*\bbT (\widetilde{\bbT}_n(z))^{-1}\bbr_j\bbr_j^*\bbD_{1j}^{-1}(z)\bbe_i \right. \\
	\nonumber &\quad\quad\quad \left. -n^{-1}\beta_{j1}\bbe_i^* \bbT(\widetilde{\bbT}_n(z))^{-1}\bE \bbD_1^{-1}(z)\bbe_i \right]+\bigo(n^{-1}) \\
	\nonumber =& \tau_1+\tau_2+\tau_3+\bigo(n^{-1}),
\end{align}
where 
\begin{align*}
	\tau_1 & =\sum_{j=2}^{n}\bE\left[ \beta_{j1}b_{j1}\widetilde{\epsilon}_j(z)\gamma_{j1}(z)\right], \\
	\tau_2 &=\frac{1}{n}\sum_{j=2}^{n}\bE \left[\beta_{j1}\bbe_i^*\bbT(\widetilde{\bbT}_n(z))^{-1}( \bbD_{1j}^{-1}(z)-\bbD_{1}^{-1}(z))\bbe_i \right], \\
	\tau_3 &=\frac{1}{n}\sum_{j=2}^{n}\bE\left[\beta_{j1}\bbe_i^*\bbT(\widetilde{\bbT}_n(z))^{-1}(\bbD_{1}^{-1}(z)-\bE \bbD_{1}^{-1}(z))\bbe_i \right].
\end{align*}
According to \eqref{beta_ij decomposition}, \eqref{D_j^{-1}-D_kj^{-1}}, Lemmas \ref{bound spectral norm of zI-bj(z)T}, \ref{a(v)_quadratic_minus_trace}, \ref{E|trTDj-1-EtrTDj-1|}, \ref{quadratic form minus trace of qudratic form}, we have $|\tau_1|=\bigo(n^{-1})$, $|\tau_2|=\bigo(n^{-1})$, $|\tau_3|=\bigo(n^{-1})$. 

Thus, $\eqref{eiTED-Tnei}=\bigo(n^{-1})$. Further, 
\begin{align*}
	&\left| n^{-2}\sum_{j=1}^{n}z_1z_2\underline{s}_n^0(z_1)\underline{s}_n^0(z_2) \operatorname{tr}[ (\bbT\mathbb{E}\bbD_j^{-1}(z_1))\circ(\bbT\mathbb{E}\bbD_j^{-1}(z_2))]\right. \\
	&\quad \left. -n^{-1}z_1z_2\underline{s}_n^0(z_1)\underline{s}_n^0(z_2) \operatorname{tr}[ (\bbT(-\tilde{\bbT}_n(z_1))^{-1})\circ(\bbT(-\tilde{\bbT}_n(z_2))^{-1}))]\right|=\bigo(n^{-1}).
\end{align*}
Combining \eqref{EjDj-EDj}, the limit of \eqref{hadmardprod} is shown that
\begin{align}\label{resulthadmard}
	&\bE\left|\eqref{hadmardprod}-n^{-1}z_1z_2\underline{s}_n^0(z_1)\underline{s}_n^0(z_2) \operatorname{tr}[ (\bbT(-\tilde{\bbT}_n(z_1))^{-1})\circ(\bbT(-\tilde{\bbT}_n(z_2))^{-1}))] \right|\\
	\nonumber &=\bigo(n^{-1/2}).	
\end{align}

Next, we will find the limit of \eqref{EtrTE_jD_jTE_jD_j}. From Lemma \ref{E|trTDj-1-EtrTDj-1|}, we have
\begin{align}
	& \mathbb{E}\lvert n^{-1}\mathbb{E}_j\operatorname{tr}\bbT\bbD_j^{-1}(z_i)-n^{-1}\operatorname{tr}\bbT\bbD_j^{-1}(z_i)\rvert \\
	\nonumber =&\mathbb{E}\lvert n^{-1}\mathbb{E}_j\operatorname{tr}\bbT\bbD_j^{-1}(z_i)-(\widetilde{\beta}_j(z_i)^{-1}-1)\rvert \\
	\nonumber=&n^{-1}\mathbb{E}\left| \sum_{k=j+1}^{n}(\mathbb{E}_k-\mathbb{E}_{k-1})\operatorname{tr}\bbT\bbD_j^{-1}(z_i) \right|  \\
	\nonumber=&n^{-1}\mathbb{E}\left| \sum_{k=j+1}^{n}(\mathbb{E}_k-\mathbb{E}_{k-1})\operatorname{tr}\bbT(\bbD_j^{-1}(z_i)-\bbD_{kj}^{-1}(z_i)) \right| 
	\leq Kn^{-1}, 
\end{align}
therefore from \eqref{E|tilde beta_j(z)+zs_n^0(z)|}, we get
\begin{align}\label{E|n^-1E_jtrTD_j^-1-(-(z_is_n^0)^-1-1)|}
	\mathbb{E}\lvert n^{-1}\mathbb{E}_j\operatorname{tr}\bbT\bbD_j^{-1}(z_i)-(-(z_i\underline{s}_n^0(z_i))^{-1}-1)\rvert\leq Kn^{-1}.
\end{align}
Recall the definitions $\bbD_{ij}(z)=\bbD(z)-\bbr_i\bbr_i^{*}-\bbr_j\bbr_j^{*}$, $\beta_{ij}(z)=(1+\bbr_i^{*}\bbD_{ij}^{-1}(z)\bbr_i)^{-1}$ and $b_{ij}(z)=(1+n^{-1}\mathbb{E}\operatorname{tr}\bbT\bbD_{ij}^{-1}(z))^{-1}$.
Similar to Lemma \ref{|b_j-b|} we have
\begin{align}\label{b_j-b_ij}
	\lvert b_j(z_i)-b_{ij}(z_i)\rvert\leq Kn^{-1}.
\end{align}
Using Lemma \ref{Q_QMTr} and a similar argument resulting in Lemma \ref{E|beta_j-b_j|}, we find
\begin{align}\label{E|beta_ij-b_ij|}
	\mathbb{E}\lvert \beta_{ij}(z_i)-b_{ij}(z_i)\rvert^2\leq Kn^{-1}.
\end{align}
Define $ \breve{\bbD}_j(z) $ to have the same structure as the matrix $ \bbD_j(z) $ with random
vectors $ r_{j+1},\dots, r_n $ replaced by their i.i.d. copies $ \breve{r}_{j+1},\dots, \breve{r}_n $. Define $ \breve{\bbD}_{ij} $ and $ \breve{\beta}_{ij} $ accordingly.
Then part of \eqref{EtrTE_jD_jTE_jD_j} becomes
\begin{align}
	\frac{1}{n^2}\sum_{j=1}^{n}z_1z_2\underline{s}_n^0(z_1)\underline{s}_n^0(z_2)\operatorname{tr} \bbT\mathbb{E}_j(\bbD_j^{-1}(z_1)\bbT\breve{\bbD}_j^{-1}(z_2)).
\end{align}
We write
\begin{align}\label{construct an equation 1}
	&\frac{z_1}{n}\mathbb{E}_j \operatorname{tr}\bbT\bbD_j^{-1}(z_1)-\frac{z_2}{n}\mathbb{E}_j\operatorname{tr} \bbT\breve{\bbD}_j^{-1}(z_2)=z_2-z_1+(\underline{s}_n^0(z_2))^{-1}-(\underline{s}_n^0(z_1))^{-1}+A_{1,n},
\end{align}
where from \eqref{E|n^-1E_jtrTD_j^-1-(-(z_is_n^0)^-1-1)|} we have $\mathbb{E}\lvert A_{1,n}\rvert\leq Kn^{-1/2}$. On the other hand,  the left-hand side of (\ref{construct an equation 1}) can be written as 
\begin{align}\label{construct an equation 2}
	& \frac{z_1}{n}\mathbb{E}_j\operatorname{tr}\bbT\bbD_j^{-1}(z_1)-\frac{z_2}{n}\mathbb{E}_j\operatorname{tr} \bbT\breve{\bbD}_j^{-1}(z_2) \\
	\nonumber=&\frac{1}{n}\mathbb{E}_j\operatorname{tr} \bbT \bbD_j^{-1}(z_1)(z_1\breve{\bbD}_j(z_2)-z_2\bbD_j(z_1))\breve{\bbD}_j^{-1}(z_2) \\
	\nonumber=&\frac{1}{n}\sum_{i=1}^{j-1}(z_1-z_2)\mathbb{E}_j \bbr_i^{*}\breve{\bbD}_{ij}^{-1}(z_2)\bbT\bbD_{ij}^{-1}(z_1)\bbr_i\beta_{ij}(z_1)\breve{\beta}_{ij}(z_2) \\
	\nonumber&+ \frac{1}{n}\sum_{i=j+1}^{n}z_1\mathbb{E}_j\breve{\bbr}_i^{*}\breve{\bbD}_{ij}^{-1}(z_2)\bbT\bbD_j^{-1}(z_1)\breve{\bbr}_i\breve{\beta}_{ij}(z_2) \\
	\nonumber&-\frac{1}{n}\sum_{i=j+1}^{n} z_2\mathbb{E}_j\bbr_i^{*}\breve{\bbD}_j^{-1}(z_2)\bbT\bbD_{ij}^{-1}(z_1)\bbr_i\beta_{ij}(z_1) \\
	\nonumber=&\frac{1}{n^2}\sum_{i=1}^{j-1}(z_1-z_2)b_{ij}(z_1)b_{ij}(z_2)\operatorname{tr}\mathbb{E}_j\bbT\breve{\bbD}_{ij}^{-1}(z_2)\bbT\bbD_{ij}^{-1}(z_1) \\
	\nonumber&+\frac{1}{n^2}\sum_{i=j+1}^{n}z_1b_{ij}(z_2)\operatorname{tr}\mathbb{E}_j\bbT\breve{\bbD}_{ij}^{-1}(z_2)\bbT\bbD_{j}^{-1}(z_1) \\
	\nonumber&-\frac{1}{n^2}\sum_{i=j+1}^{n}z_2b_{ij}(z_1)\operatorname{tr}\mathbb{E}_j\bbT\breve{\bbD}_j^{-1}(z_2)\bbT\bbD_{ij}^{-1}(z_1)+A_{2,n} \\
	\nonumber=&\frac{1}{n^2}\sum_{i=1}^{j-1}(z_1-z_2)z_1z_2\underline{s}_n^0(z_1)\underline{s}_n^0(z_2)\operatorname{tr}\mathbb{E}_j\bbT\breve{\bbD}_{ij}^{-1}(z_2)\bbT\bbD_{ij}^{-1}(z_1) \\
	\nonumber &-\frac{1}{n^2}\sum_{i=j+1}^{n}z_1z_2\underline{s}_n^0(z_2)\operatorname{tr}\mathbb{E}_j\bbT\breve{\bbD}_{ij}^{-1}(z_2)\bbT\bbD_{j}^{-1}(z_1) \\
	\nonumber&+\frac{1}{n^2}\sum_{i=j+1}^{n}z_1z_2\underline{s}_n^0(z_1)\operatorname{tr}\mathbb{E}_j\bbT\breve{\bbD}_j^{-1}(z_2)\bbT\bbD_{ij}^{-1}(z_1)+A_{3,n}.
\end{align}
From \eqref{beta_ij decomposition},  \eqref{E|beta_ij-b_ij|}, Lemmas \ref{bound moment of b_j and beta_j} and  \ref{a(v)_quadratic_minus_trace}, we have $\mathbb{E}\lvert A_{2,n}\rvert\leq Kn^{-1/2}$. Further, from \eqref{b_j(z)+zs_n^0(z)} and \eqref{b_j-b_ij}, we have $\mathbb{E}\lvert A_{3.n}\rvert\leq Kn^{-1/2}$.
Using \eqref{D_j^{-1}-D_kj^{-1}}, we find 
\begin{align}\label{construct an equation 3}
	\eqref{construct an equation 2} & =\frac{j-1}{n^2}(z_1-z_2)z_1 z_2\underline{s}_n^0(z_1)\underline{s}_n^0(z_2)\operatorname{tr} \bbT\mathbb{E}_j(\bbD_j^{-1}(z_1))\bbT\mathbb{E}_j(\breve{\bbD}_j^{-1}(z_2)) \\
	\nonumber&\quad - \frac{n-j}{n^2}z_1 z_2\underline{s}_n^0(z_2)\operatorname{tr}\bbT\mathbb{E}_j(\bbD_j^{-1}(z_1))\bbT\mathbb{E}_j(\breve{\bbD}_j^{-1}(z_2)) \\
	\nonumber&\quad +\frac{n-j}{n^2}z_1 z_2\underline{s}_n^0(z_1)\operatorname{tr}\bbT\mathbb{E}_j(\bbD_j^{-1}(z_1))\bbT\mathbb{E}_j(\breve{\bbD}_j^{-1}(z_2))+A_{4,n},
\end{align}
where $\mathbb{E}\lvert A_{4,n}\rvert\leq Kn^{-1/2}$.
Combining \eqref{construct an equation 1} and \eqref{construct an equation 3},  we have
\begin{align}
	& n^{-2}\sum_{j=1}^{n}z_1z_2\underline{s}_n^0(z_1)\underline{s}_n^0(z_2)\operatorname{tr}\bbT\mathbb{E}_j(\bbD_j^{-1}(z_1))\bbT\mathbb{E}_j(\breve{\bbD}_j^{-1}(z_2)) \\
	\nonumber=&\frac{n^{-2}\sum_{j=1}^{n}z_1 z_2 \underline{s}_n^0(z_2) \underline{s}_n^0(z_1)(z_2-z_1+(\underline{s}_n^{0}(z_2))^{-1}-(\underline{s}_n^{0}(z_1))^{-1})}{\frac{j-1}{n^2}(z_1-z_2) z_1 z_2 \underline{s}_n^0(z_1) \underline{s}_n^0(z_2)+\frac{n-(j-1)}{n^2} z_1 z_2(\underline{s}_n^0(z_1)-\underline{s}_n^0(z_2))} \\
	\nonumber& \quad +A_{5,n} \\
	\nonumber =&\frac{1}{n}\sum_{j=1}^{n} \frac{a_n(z_1, z_2)}{1-\frac{j-1}{n} a_n(z_1, z_2)}+A_{5, n},
\end{align}
where $\mathbb{E}\left|A_{5,n} \right| \leq Kn^{-1/2}$, and
\begin{align}
	a_n(z_1,z_2)=y_n\underline{s}_n^0(z_1)\underline{s}_n^0(z_2)\int\frac{t^2dH_p(t)}{(1+t\underline{s}_n^0(z_1))(1+t\underline{s}_n^0(z_2))}.
\end{align}
     	We have, due to $\sup_{z_1\in\mathcal{C}_2,z_2\in\mathcal{C}_1}\lvert a_n(z_1,z_2)\rvert<1$, then
     \begin{align}\label{a_n(z_1,z_2)-integral}
     	&\sup_{z_1\in\mathcal{C}_2,z_2\in\mathcal{C}_1}\left\lvert a_n(z_1,z_2)\frac{1}{n}\sum_{j=1}^{n}\frac{1}{1-\frac{j-1}{n}a_n(z_1,z_2)}-a_n(z_1,z_2)\int_{0}^{1}\frac{1}{1-ta_n(z_1,z_2)}dt\right\rvert \leq Kn^{-1}. 
     \end{align}
   Therefore,  we get the limit of \eqref{EtrTE_jD_jTE_jD_j}
     \begin{align}\label{resultproduct}
     	& \left| \frac{1}{n^2}\sum_{j=1}^{n}z_1z_2\underline{s}_n^0(z_1)\underline{s}_n^0(z_2)\operatorname{tr} \bbT\mathbb{E}_j(\bbD_j^{-1}(z_1))\bbT\mathbb{E}_j(\bbD_j^{-1}(z_2)) \right.\\
     	\nonumber &\quad\quad\quad\quad\quad\quad \left. -a_n(z_1,z_2)\int_{0}^{1}\frac{1}{1-ta_n(z_1,z_2)}dt \right| \\
     	\nonumber &\leq Kn^{-1/2}.
     \end{align}

     From \eqref{resulthadmard}, \eqref{resultproduct} and Lemma \ref{|Es_n-s_n^0|}, we get the limit of \eqref{Ejejz1Ejejz2}, which is under Assumption \ref{assumRG},
     \begin{align}\label{RGinvar}
     	\bE & \left|\eqref{Ejejz1Ejejz2}-\beta_x n^{-1}\underline{s}_n^0(z_1)\underline{s}_n^0(z_2) \sum_{i=1}^{p}\bbe_i^* \bbT(\underline{s}_n^0(z_1)\bbT+\bbI)^{-1}\bbe_i \bbe_i^*\bbT(\underline{s}_n^0(z_2)\bbT+\bbI)^{-1}\bbe_i \right.\\
     	\nonumber &~~ \left. -2a_n(z_1,z_2)\int_{0}^{1}\frac{1}{1-ta_n(z_1,z_2)}dt\right|=\bigo(n^{-1/2}),
     \end{align}
     under Assumption \ref{assumCG},
     \begin{align}\label{CGinvar}
     \bE & \left|\eqref{Ejejz1Ejejz2}-\beta_x n^{-1}\underline{s}_n^0(z_1)\underline{s}_n^0(z_2) \sum_{i=1}^{p}\bbe_i^*\bbT(\underline{s}_n^0(z_1)\bbT+\bbI)^{-1}\bbe_i\bbe_i^*\bbT(\underline{s}_n^0(z_2)\bbT+\bbI)^{-1}\bbe_i \right.	\\
     \nonumber&~~ \left. -a_n(z_1,z_2)\int_{0}^{1}\frac{1}{1-ta_n(z_1,z_2)}dt\right|=\bigo(n^{-1/2}).
     \end{align}
	As $n\to \infty$, $y_n\to y$, $\underline{s}_n^0(z)\to \underline{s}^0(z)$, $f_m^{\prime}(z)\to f^{\prime}(z)$, $g_m^{\prime}(z)\to g^{\prime}(z)$, $H_p\to H$,
	$$p^{-1} \sum_{i=1}^{p}\bbe_i^*\bbT(\underline{s}_n^0(z_1)\bbT+\bbI)^{-1}\bbe_i\bbe_i^*\bbT(\underline{s}_n^0(z_2)\bbT+\bbI)^{-1}\bbe_i\rightarrow h_1(z_1,z_2),$$
	 we have under Assumption \ref{assumRG}, 
	\begin{align*}
		&-\frac{1}{4\pi^2}\oint_{\mathcal{C}_{1}}\oint_{\mathcal{C}_{2}}f_m^{\prime}(z_1)g_m^{\prime}(z_2)\sum_{j=1}^{n}\mathbb{E}_{j-1}[  \mathbb{E}_j\varepsilon_j(z_1)b_j(z_1)  \mathbb{E}_j\varepsilon_j(z_2)b_j(z_2)]dz_1dz_2\\
	& \rightarrow -\frac{y\beta_x}{4\pi^2}\oint_{\mathcal{C}_{1}}\oint_{\mathcal{C}_{2}}f^{\prime}(z_1)g^{\prime}(z_2) \underline{s}^0(z_1)\underline{s}^0(z_2)h_1(z_1,z_2)dz_2dz_1 \\
	&\quad\quad\quad -\frac{1}{2\pi^2}\oint_{\mathcal{C}_{1}}\oint_{\mathcal{C}_{2}}f^{\prime}(z_1)g^{\prime}(z_2) a(z_1,z_2)\int_{0}^{1}\frac{1}{1-ta(z_1,z_2)}dt dz_2dz_1;
	\end{align*}
	under Assumption \ref{assumCG},
	\begin{align*}
		&-\frac{1}{4\pi^2}\oint_{\mathcal{C}_{1}}\oint_{\mathcal{C}_{2}}f_m^{\prime}(z_1)g_m^{\prime}(z_2)\sum_{j=1}^{n}\mathbb{E}_{j-1}[  \mathbb{E}_j\varepsilon_j(z_1)b_j(z_1)  \mathbb{E}_j\varepsilon_j(z_2)b_j(z_2)]dz_1dz_2\\
	& \rightarrow -\frac{y\beta_x}{4\pi^2}\oint_{\mathcal{C}_{1}}\oint_{\mathcal{C}_{2}}f^{\prime}(z_1)g^{\prime}(z_2)  \underline{s}^0(z_1)\underline{s}^0(z_2)h_1(z_1,z_2)dz_2dz_1 \\
	&\quad\quad\quad -\frac{1}{4\pi^2}\oint_{\mathcal{C}_{1}}\oint_{\mathcal{C}_{2}}f^{\prime}(z_1)g^{\prime}(z_2) a(z_1,z_2)\int_{0}^{1}\frac{1}{1-ta(z_1,z_2)}dt dz_2dz_1.
	\end{align*}
	Thus, we can get the covariance term of Theorem \ref{FunctionCLT}.                 
	
\end{proof}

\subsection{Nonrandom part}
In this section, we will find the limit of 
\begin{align*}
	 p\int f_m(x)d[\mathbb{E}F^{B_n}(x)-F^{y_n,H_p}(x)]=-\frac{1}{2\pi i}\oint_{\mathcal{C}}f_m(z)M_n^2(z)dz.
\end{align*}
To this end, we shall first consider $M_n^2(z)=p\left[ \bE s_n(z)-s_n^0(z) \right]=n\left[\bE \underline{s}_n(z)-\underline{s}_n^0(z)\right]$.

Let $$A_n(z)=y_n\int \frac{dH_p(t)}{1+t\mathbb{E}\underline{s}_n(z)}+z y_n\mathbb{E}s_n(z).$$ Using the identity
\begin{align*}
	\mathbb{E}\underline{s}_n(z)=-\frac{1-y_n}{z}+y_n\mathbb{E}s_n
\end{align*}
we have 
\begin{align*}
	A_n(z) &= y_n\int \frac{dH_p(t)}{1+t\mathbb{E}\underline{s}_n(z)}-y_n+z\mathbb{E}\underline{s}_n(z)+1 \\
	\nonumber &=-\mathbb{E}\underline{s}_n(z)\left(-z-\frac{1}{\mathbb{E}\underline{s}_n(z)}+y_n\int \frac{tdH_p(t)}{1+t\mathbb{E}\underline{s}_n(z)}\right).
\end{align*}
It follows that
\begin{align}\label{Es_n(z)}
	\mathbb{E}\underline{s}_n(z)=\left[ -z+y_n\int \frac{tdH_p(t)}{1+t\mathbb{E}\underline{s}_n(z)}+A_n/\mathbb{E}\underline{s}_n(z)\right]^{-1}.
\end{align}

From \eqref{equation of s^0}, we have equation
\begin{align}\label{equation of s_n^0}
	\underline{s}_n^0=-\left(z-y_n\int\frac{t}{1+t\underline{s}_n^0}dH_p(t)\right)^{-1},
\end{align}
where $\underline{s}_n^0:=s_{\underline{F}^{y_n,H_p}}$ and $\underline{F}^{y_n,H_p}=(1-y_n)\delta_0+y_nF^{y_n,H_p}$.
From \eqref{equation of s_n^0} and \eqref{Es_n(z)}, we write
\begin{align}\label{M_n^2}
	& M_n^2(z)= n\left[\frac{1}{-z+y_n\int \frac{tdH_p(t)}{1+t\mathbb{E}\underline{s}_n(z)}+A_n/\mathbb{E}\underline{s}_n(z)}-\frac{1}{-z+y_n\int \frac{tdH_p(t)}{1+t\underline{s}_n^0(z)}}\right] \\
	\nonumber=& -n\underline{s}_n^0 A_n\left[1-y_n\mathbb{E}\underline{s}_n \underline{s}_n^0\int\frac{t^2dH_p(t)}{(1+t\mathbb{E}\underline{s}_n)(1+t\underline{s}_n^0)}\right]^{-1}.
\end{align}
Due to (4.4) in \cite{BaiS04C}, we can get the existence of $\xi\in(0,1)$ such that for all $n$ large 
\begin{align*}
	\sup_{z\in\mathcal{C}}\bigg\lvert y_n\mathbb{E}\underline{s}_n \underline{s}_n^0\int\frac{t^2dH_p(t)}{(1+t\mathbb{E}\underline{s}_n)(1+t\underline{s}_n^0)}\bigg\rvert<\xi.
\end{align*}
Thus the denominator of \eqref{M_n^2} is bounded away from zero.
Our next task is to investigate the limiting behavior of
\begin{align*}
	 nA_n&=n\left(y_n\int \frac{dH_p(t)}{1+t\mathbb{E}\underline{s}_n(z)}+zy_n\mathbb{E}s_n(z)\right) \\
	&=n\mathbb{E}\beta_1\left[\bbr_1^{*}\bbD_1^{-1}(\mathbb{E}\underline{s}_n\bbT+\bbI)^{-1}\bbr_1-n^{-1}\mathbb{E}\operatorname{tr}(\mathbb{E}\underline{s}_n\bbT+\bbI)^{-1}\bbT\bbD^{-1}\right],
\end{align*}
for $z\in\mathcal{C}$ [see (5.2) in \cite{BaiS98N}]. Throughout the following, all bounds, including $\bigo(\cdot)$ and $o(\cdot)$ expressions hold uniformly for $z\in\mathcal{C}$. The positive constant $K$ is independent of $z$.  
According to \eqref{D^{-1}-D_j^{-1}}, \eqref{beta_1}, Lemmas \ref{a(v)_quadratic_minus_trace} and \ref{E|trTDj-1-EtrTDj-1|}, we have 
\begin{align}
	& \lvert\mathbb{E}\operatorname{tr}(\mathbb{E}\underline{s}_n\bbT+\bbI)^{-1}\bbT\bbD_1^{-1}-\mathbb{E}\operatorname{tr}(\mathbb{E}\underline{s}_n\bbT+\bbI)^{-1}\bbT\bbD^{-1} \\
	\nonumber&\quad\quad\quad\quad\quad\quad\quad\quad\quad\quad\quad\quad\quad\quad\quad-b_n\mathbb{E}\bbr_1^{*}\bbD_1^{-1}(\mathbb{E}\underline{s}_n\bbT+\bbI)^{-1}\bbT\bbD_1^{-1}\bbr_1\rvert \\
	\nonumber =&\lvert\mathbb{E}\beta_1 \operatorname{tr}(\mathbb{E}\underline{s}_n\bbT+\bbI)^{-1}\bbT\bbD_1^{-1}\bbr_1\bbr_1^{*}\bbD_1^{-1}-b_n\mathbb{E}\bbr_1^{*}\bbD_1^{-1}(\mathbb{E}\underline{s}_n\bbT+\bbI)^{-1}\bbT\bbD_1^{-1}\bbr_1\rvert \\
	\nonumber =&\lvert b_n\mathbb{E}(1-\beta_1\gamma_1)\bbr_1^{*}\bbD_1^{-1}(\mathbb{E}\underline{s}_n\bbT+\bbI)^{-1}\bbT\bbD_1^{-1}\bbr_1 \\
	\nonumber &\quad\quad\quad\quad\quad\quad\quad\quad\quad\quad\quad\quad\quad\quad\quad-b_n\mathbb{E}\bbr_1^{*}\bbD_1^{-1}(\mathbb{E}\underline{s}_n\bbT+\bbI)^{-1}\bbT\bbD_1^{-1}\bbr_1\rvert \\
	\nonumber =&\lvert b_n\mathbb{E}\beta_1\gamma_1\bbr_1^{*}\bbD_1^{-1}(\mathbb{E}\underline{s}_n\bbT+\bbI)^{-1}\bbT\bbD_1^{-1}\bbr_1\rvert \leq Kn^{-1}. 
\end{align}
Since \eqref{beta_1}, we have
\begin{align*}
	& n\mathbb{E}\beta_1\bbr_1^{*}\bbD_1^{-1}(\mathbb{E}\underline{s}_n\bbT+\bbI)^{-1}\bbr_1-\mathbb{E}\beta_1\mathbb{E}\operatorname{tr}(\mathbb{E}\underline{s}_n\bbT+\bbI)^{-1}\bbT\bbD_1^{-1} \\
	=&-b_n^2 n\mathbb{E}\gamma_1\bbr_1^{*}\bbD_1^{-1}(\mathbb{E}\underline{s}_n\bbT+\bbI)^{-1}\bbr_1  +b_n^2\left[n\mathbb{E}\beta_1\gamma_1^2\bbr_1^{*}\bbD_1^{-1}(\mathbb{E}\underline{s}_n\bbT+\bbI)^{-1}\bbr_1 \right. \\
	& \quad\quad\quad\quad \left.-(\mathbb{E}\beta_1\gamma_1^2)\mathbb{E}\operatorname{tr}(\mathbb{E}\underline{s}_n\bbT+\bbI)^{-1}\bbT\bbD_1^{-1}\right] \\
	=&-b_n^2n\mathbb{E}\gamma_1\bbr_1^{*}\bbD_1^{-1}(\mathbb{E}\underline{s}_n\bbT+\bbI)^{-1}\bbr_1 +b_n^2 \operatorname{Cov}(\beta_1\gamma_1^2,\operatorname{tr}\bbD_1^{-1}(\mathbb{E}\underline{s}_n\bbT+\bbI)^{-1}\bbT) \\
	& +b_n^2\mathbb{E}\left[n\beta_1\gamma_1^2\bbr_1^{*}\bbD_1^{-1}(\mathbb{E}\underline{s}_n\bbT+\bbI)^{-1}\bbr_1-\beta_1\gamma_1^2\operatorname{tr}\bbD_1^{-1}(\mathbb{E}\underline{s}_n\bbT+\bbI)^{-1}\bbT\right].
\end{align*}
Here $\operatorname{Cov}(X,Y)=\mathbb{E}XY-\mathbb{E}X\mathbb{E}Y$.
Using (4.3) in \cite{BaiS04C},  \eqref{beta_1}, Lemmas \ref{bound moment of b_j and beta_j}, \ref{a(v)_quadratic_minus_trace}, \ref{lem3Qua} and \ref{E|trTDj-1-EtrTDj-1|},  we have 
\begin{align*}
	& \left|\mathbb{E}\left[n\beta_1\gamma_1^2\bbr_1^{*}\bbD_1^{-1}(\mathbb{E}\underline{s}_n\bbT+\bbI)^{-1}\bbr_1-\beta_1\gamma_1^2\operatorname{tr}\bbD_1^{-1}(\mathbb{E}\underline{s}_n\bbT+\bbI)^{-1}\bbT\right]\right| \\
	=&\left| nb_n\mathbb{E}\left[\gamma_1^2(\bbr_1^{*}\bbD_1^{-1}(\mathbb{E}\underline{s}_n\bbT+\bbI)^{-1}\bbr_1-n^{-1}\operatorname{tr}\bbD_1^{-1}(\mathbb{E}\underline{s}_n\bbT+\bbI)^{-1}\bbT)\right] \right. \\
	&\quad- nb_n^2\mathbb{E}\left[\gamma_1^3(\bbr_1^{*}\bbD_1^{-1}(\mathbb{E}\underline{s}_n\bbT+\bbI)^{-1}\bbr_1-n^{-1}\operatorname{tr}\bbD_1^{-1}(\mathbb{E}\underline{s}_n\bbT+\bbI)^{-1}\bbT)\right] \\
	&\quad \left. +b_n^2\mathbb{E}\left[\beta_1\gamma_1^4(\bbr_1^{*}\bbD_1^{-1}(\mathbb{E}\underline{s}_n\bbT+\bbI)^{-1}\bbr_1-n^{-1}\operatorname{tr}\bbD_1^{-1}(\mathbb{E}\underline{s}_n\bbT+\bbI)^{-1}\bbT)\right]\right| \\
	\leq& Kn\left| \mathbb{E}(\bbr_1^{*}\bbD_1^{-1}\bbr_1-n^{-1}\operatorname{tr}\bbT\bbD_1^{-1})^2(\bbr_1^{*}\bbD_1^{-1}(\mathbb{E}\underline{s}_n\bbT+\bbI)^{-1}\bbr_1-n^{-1}\operatorname{tr}\bbD_1^{-1}(\mathbb{E}\underline{s}_n\bbT+\bbI)^{-1}\bbT)\right| \\
	& +Kn\sqrt{\mathbb{E}\lvert\gamma_1\rvert^6}\sqrt{\mathbb{E}\lvert \bbr_1^{*}\bbD_1^{-1}(\mathbb{E}\underline{s}_n\bbT+\bbI)^{-1}\bbr_1-n^{-1}\operatorname{tr}\bbD_1^{-1}(\mathbb{E}\underline{s}_n\bbT+\bbI)^{-1}\bbT\rvert^2}  \\
	& +Kn\sqrt{\mathbb{E}\lvert \gamma_1\rvert^8}\sqrt{\mathbb{E}\lvert \beta_1(\bbr_1^{*}\bbD_1^{-1}(\mathbb{E}\underline{s}_n\bbT+\bbI)^{-1}\bbr_1-n^{-1}\operatorname{tr}\bbD_1^{-1}(\mathbb{E}\underline{s}_n\bbT+\bbI)^{-1}\bbT)\rvert^2}  \\
	\leq& Kn^{-1}+Kn^{-1}+Kn^{-3/2} \leq Kn^{-1}.
\end{align*} 
Using (4.3) in \cite{BaiS04C}, Lemmas \ref{bound moment of b_j and beta_j}, \ref{a(v)_quadratic_minus_trace} and \ref{E|trTDj-1-EtrTDj-1|}, we have
\begin{align*}
	&  \lvert\operatorname{Cov}(\beta_1\gamma_1^2,\operatorname{tr}\bbD_1^{-1}(\mathbb{E}\underline{s}_n\bbT+\bbI)^{-1}\bbT)\rvert \\
	\leq& (\mathbb{E}\lvert\beta_1\rvert^4)^{1/4}(\mathbb{E}\lvert\gamma_1\rvert^8)^{1/4} (\mathbb{E}\lvert \operatorname{tr}\bbD_1^{-1}(\mathbb{E}\underline{s}_n\bbT+\bbI)^{-1}\bbT-\mathbb{E}\operatorname{tr}\bbD_1^{-1}(\mathbb{E}\underline{s}_n\bbT+\bbI)^{-1}\bbT\rvert^2)^{1/2} \\
	\leq& Kn^{-1}.
\end{align*}
Write
\begin{align*}
	&  n\mathbb{E}\gamma_1\bbr_1^{*}\bbD_1^{-1}(\mathbb{E}\underline{s}_n\bbT+\bbI)^{-1}\bbr_1 \\
	=& n\mathbb{E}[(\bbr_1^{*}\bbD_1^{-1}(\mathbb{E}\underline{s}_n\bbT+\bbI)^{-1}\bbr_1-n^{-1}\operatorname{tr}\bbD_1^{-1}(\mathbb{E}\underline{s}_n\bbT+\bbI)^{-1}\bbT)  (\bbr_1^{*}\bbD_1^{-1}\bbr_1-n^{-1}\operatorname{tr}\bbD_1^{-1}\bbT)] \\
	&\quad +n^{-1}\operatorname{Cov}(\operatorname{tr}\bbD_1^{-1}\bbT,\operatorname{tr}\bbD_1^{-1}(\mathbb{E}\underline{s}_n\bbT+\bbI)^{-1}\bbT).
\end{align*}
From Lemma \ref{E|trTDj-1-EtrTDj-1|} we see the second term above is $\bigo(n^{-1})$. Since \eqref{|b_j-Ebeta_j|}, we get $\mathbb{E}\beta_1=b_n+\bigo(n^{-1})$.
Therefore, we arrive at
\begin{align}
	& nA_n=n\left(y_n\int \frac{dH_p(t)}{1+t\mathbb{E}\underline{s}_n(z)}+zy_n\mathbb{E}s_n(z)\right) \\
	\nonumber=& b_n^2 n^{-1}\mathbb{E}\operatorname{tr}\bbD_1^{-1}(\mathbb{E}\underline{s}_n\bbT+\bbI)^{-1}\bbT\bbD_1^{-1}\bbT \\
	\nonumber&\quad -b_n^2n\mathbb{E}[(\bbr_1^{*}\bbD_1^{-1}(\mathbb{E}\underline{s}_n\bbT+\bbI)^{-1}\bbr_1-n^{-1}\operatorname{tr}\bbD_1^{-1}(\mathbb{E}\underline{s}_n\bbT+\bbI)^{-1}\bbT)  \\
	\nonumber &\quad\quad\quad\quad\quad \times (\bbr_1^{*}\bbD_1^{-1}\bbr_1-n^{-1}\operatorname{tr}\bbD_1^{-1}\bbT)]+\bigo(n^{-1}).
\end{align}
Due to \eqref{b_j(z)+zs_n^0(z)} and Lemma \ref{product of two quadratic minus trace of two matrices}, we see that under Assumption \ref{assumCG}
\begin{align}\label{nA_n in CG case}
	nA_n=-(z\underline{s}_n^0(z))^2n^{-1}\beta_x\bE\left[\operatorname{tr}(\bbD_1^{-1}(\bE\underline{s}_n(z)\bbT+\bbI)^{-1}\bbT)\circ(\bbD_1^{-1}\bbT)\right]+\bigo(n^{-1}),
\end{align} 
while under Assumption \ref{assumRG},
\begin{align}
	nA_n= & (z\underline{s}_n^0(z))^2 n^{-1}\mathbb{E}\operatorname{tr}\bbD_1^{-1}(\mathbb{E}\underline{s}_n\bbT+\bbI)^{-1}\bbT\bbD_1^{-1}\bbT \\
	\nonumber &-(z\underline{s}_n^0(z))^2n^{-1}\beta_x\bE\left[\operatorname{tr}(\bbD_1^{-1}(\bE\underline{s}_n(z)\bbT+\bbI)^{-1}\bbT)\circ(\bbD_1^{-1}\bbT)\right]+\bigo(n^{-1}).
\end{align}
We only need to estimate the limits of 
\begin{align}\label{meanlimit1}
	n^{-1}\bE\left[\operatorname{tr}(\bbD_1^{-1}(\bE\underline{s}_n(z)\bbT+\bbI)^{-1}\bbT)\circ(\bbD_1^{-1}\bbT)\right] 
\end{align}
and 
\begin{align}\label{meanlimit2}
	n^{-1}\mathbb{E}\left[\operatorname{tr}\bbD_1^{-1}(\mathbb{E}\underline{s}_n\bbT+\bbI)^{-1}\bbT\bbD_1^{-1}\bbT\right]. 
\end{align}

As for the limit of \eqref{meanlimit1}, due to Lemmas \ref{Burkholder_1}, \ref{bound moment of b_j and beta_j}, \ref{bound spectral norm of zI-bj(z)T}, \ref{QMTr_1}, \eqref{D_j^{-1}-D_kj^{-1}},
\begin{align*}
	&\left|n^{-1}\bE\left[\operatorname{tr}(\bbD_1^{-1}(\bE\underline{s}_n(z)\bbT+\bbI)^{-1}\bbT)\circ((\bbD_1^{-1}-\bE\bbD_1^{-1})\bbT)\right]\right| \\
	\leq & n^{-1}\sum_{i=1}^{p}\left|\bE\left[\bbe_i^*(\bbD_1^{-1}-\bE\bbD_1^{-1})(\bE\underline{s}_n(z)\bbT+\bbI)^{-1}\bbT \bbe_i \bbe_i^*(\bbD_1^{-1}-\bE\bbD_1^{-1})\bbT\bbe_i \right]\right| \\
	\leq & n^{-1}\sum_{i=1}^{p}\left( \sum_{s=1}^{n}\bE|\bbr_s^*\bbD_{1s}^{-1}(\bE\underline{s}_n(z)\bbT+\bbI)^{-1}\bbT\bbe_i\bbe_i^*\bbD_{1s}^{-1}\bbr_s \right. \\
	& \left. \quad\quad\quad\quad\quad\quad\quad -n^{-1}\mtr\bbT\bbD_{1s}^{-1}(\bE\underline{s}_n(z)\bbT+\bbI)^{-1}\bbT\bbe_i\bbe_i^*\bbD_{1s}^{-1}|^2 \right)^{1/2} \\
	&\quad\quad\quad\quad \left(\sum_{s=1}^{n}\bE|\bbr_s^*\bbD_{1s}^{-1}\bbT\bbe_i\bbe_i^*\bbD_{1s}^{-1}\bbr_s-n^{-1}\mtr \bbT\bbD_{1s}^{-1}\bbT\bbe_i\bbe_i^*\bbD_{1s}^{-1}|^2\right)^{1/2} \\
	=&\bigo(n^{-1}).
\end{align*}
Thus, 
\begin{align*}
	\left|\eqref{meanlimit1}-n^{-1} \left[\mtr (\bE \bbD_1^{-1}(\bE\underline{s}_n(z)\bbT+\bbI)^{-1}\bbT)\circ (\bE\bbD_1^{-1}\bbT)\right]\right|=\bigo(n^{-1}).
\end{align*}
Further, due to $\eqref{eiTED-Tnei}=\bigo(n^{-1})$ and Lemma \ref{|Es_n-s_n^0|}, we get the limit of \eqref{meanlimit1},
\begin{align*}
	\eqref{meanlimit1}=n^{-1} z^{-2} \mtr ((\underline{s}_n^0(z)\bbT+\bbI)^{-2}\bbT)\circ ((\underline{s}_n^0(z)\bbT+\bbI)^{-1}\bbT)+\bigo(n^{-1}).
\end{align*}

Next, we find the limit of \eqref{meanlimit2}. Applications of \eqref{D^{-1}-D_j^{-1}}, (4.3) in \cite{BaiS04C}, Lemmas \ref{bound moment of b_j and beta_j}, \ref{a(v)_quadratic_minus_trace} and \ref{E|trTDj-1-EtrTDj-1|} show that both
\begin{align*}
	\mathbb{E}\operatorname{tr}\bbD_1^{-1}(\mathbb{E}\underline{s}_n\bbT+\bbI)^{-1}\bbT\bbD_1^{-1}\bbT-\mathbb{E}\operatorname{tr}\bbD^{-1}(\mathbb{E}\underline{s}_n\bbT+\bbI)^{-1}\bbT\bbD_1^{-1}\bbT
\end{align*}
and
\begin{align*}
	\mathbb{E}\operatorname{tr}\bbD^{-1}(\mathbb{E}\underline{s}_n\bbT+\bbI)^{-1}\bbT\bbD_1^{-1}\bbT-\mathbb{E}\operatorname{tr}\bbD^{-1}(\mathbb{E}\underline{s}_n\bbT+\bbI)^{-1}\bbT\bbD^{-1}\bbT
\end{align*}
are bounded. Therefore it is sufficient to consider 
\begin{align*}
	n^{-1}\mathbb{E}\operatorname{tr}\bbD^{-1}(\mathbb{E}\underline{s}_n\bbT+\bbI)^{-1}\bbT\bbD^{-1}\bbT.
\end{align*}
Write
\begin{align*}
	\bbD(z)+z\bbI-b_n(z)\bbT=\sum_{j=1}^{n}\bbr_j\bbr_j^{*}-b_n(z)\bbT.
\end{align*}
It is straightforward to verify that $z\bbI-b_n(z)\bbT$ is non-singular. Multiplying by $(z\bbI-b_n(z)\bbT)^{-1}$ on the left-hand side, $\bbD_j^{-1}(z)$ on the right-hand side and using \eqref{r_i^*(C+r_ir_i^*)^-1}, we get
\begin{align}\label{D^-1(z)}
	\bbD^{-1}(z)= & -(z \bbI-b_n(z) \bbT)^{-1} +\sum_{j=1}^n \beta_j(z)(z \bbI-b_n(z) \bbT)^{-1} \bbr_j \bbr_j^* \bbD_j^{-1}(z) \\
	\nonumber& -b_n(z)(z \bbI-b_n(z) \bbT)^{-1} \bbT \bbD^{-1}(z) \\
	\nonumber= & -(z \bbI-b_n(z) \bbT)^{-1}+b_n(z) \bbA(z)+\bbB(z)+\bbC(z)
\end{align}
where
\begin{align*}
	& \bbA(z)=\sum_{j=1}^n(z \bbI-b_n(z) \bbT)^{-1}(\bbr_j \bbr_j^*-n^{-1} \bbT) \bbD_j^{-1}(z), \\
	& \bbB(z)=\sum_{j=1}^n(\beta_j(z)-b_n(z))(z \bbI-b_n(z) \bbT)^{-1} \bbr_j \bbr_j^* \bbD_j^{-1}(z)
\end{align*}

and
\begin{align*}
	\bbC(z) & =n^{-1} b_n(z)(z \bbI-b_n(z) \bbT)^{-1} \bbT \sum_{j=1}^n(\bbD_j^{-1}(z)-\bbD^{-1}(z)) \\
	& =n^{-1} b_n(z)(z \bbI-b_n(z) \bbT)^{-1} \bbT \sum_{j=1}^n \beta_j(z) \bbD_j^{-1}(z) \bbr_j \bbr_j^* \bbD_j^{-1}(z) .
\end{align*}
From Lemma \ref{bound spectral norm of zI-bj(z)T} , it follows that $\lVert (z\bbI-b_n(z)\bbT)^{-1}\rVert$ is bounded. We have by \eqref{beta_1}, Lemmas \ref{a(v)_quadratic_minus_trace} and \ref{E|trTDj-1-EtrTDj-1|}
\begin{align}\label{E|b_n-beta_1|}
	\mathbb{E}\lvert\beta_1-b_n\rvert^2=\lvert b_n\rvert^2\mathbb{E}\lvert\beta_1\gamma_1\rvert^2\leq Kn^{-1}.
\end{align}
Let $\bbM$ be $n\times n$ matrix. From Lemmas 
\ref{bound moment of b_j and beta_j},  \ref{a(v)_quadratic_minus_trace}, \eqref{|b_j-Ebeta_j|} and \eqref{E|b_n-beta_1|} we get
\begin{align}\label{tr B(z)M}
	& \lvert n^{-1} \mathbb{E}\operatorname{tr} \bbB(z) \bbM\rvert \\
	\nonumber&=\lvert n^{-1}\sum_{j=1}^{n}\mathbb{E}[(\beta_j-b_n) \\
	\nonumber&\quad\quad \times(\bbr_j^*\bbD_j^{-1}\bbM(z\bbI-b_n(z)\bbT)^{-1}\bbr_j-n^{-1}\operatorname{tr}\bbT\bbD_j^{-1}\bbM(z\bbI-b_n(z)\bbT)^{-1}) \\
	\nonumber &\quad\quad\quad\quad\quad\quad+(\beta_j-b_n)n^{-1}\operatorname{tr}\bbT\bbD_j^{-1}\bbM(z\bbI-b_n(z)\bbT)^{-1}]\rvert \\
	\nonumber& \leq n^{-1}\sum_{j=1}^{n}\sqrt{\mathbb{E}\lvert \bbr_j^*\bbD_j^{-1}\bbM(z\bbI-b_n(z)\bbT)^{-1}\bbr_j-n^{-1}\operatorname{tr}\bbT\bbD_j^{-1}\bbM(z\bbI-b_n(z)\bbT)^{-1}\rvert^2} \\
	\nonumber&\quad\quad\quad\quad \times \sqrt{\mathbb{E}\lvert \beta_j-b_n\rvert^2}  \\
	\nonumber&\quad\quad+n^{-1}\sum_{j=1}^{n} \lvert \mathbb{E}(\beta_j-b_n)n^{-1}\operatorname{tr}\bbT\bbD_j^{-1}\bbM(z\bbI-b_n(z)\bbT)^{-1}\rvert \\
	\nonumber&\leq Kn^{-1}(\mathbb{E}\lVert \bbM\rVert^2)^{1/2}+Kn^{-1}\mathbb{E}\lVert \bbM\rVert \leq Kn^{-1}(\mathbb{E}\lVert \bbM\rVert^2)^{1/2},
\end{align}
and
\begin{align}\label{tr C(z)M}
	\lvert n^{-1} \mathbb{E} \operatorname{tr} \bbC(z) \bbM\rvert & \leq K \mathbb{E}\lvert\beta_1\rvert \bbr_1^* \bbr_1\lVert \bbD_1^{-1}\rVert^2\lVert \bbM\rVert \\
	\nonumber& \leq K n^{-1}(\mathbb{E}\lVert \bbM\rVert^2)^{1 / 2}.
\end{align}
For the following $\bbM, n \times n$ matrix, is nonrandom, bounded in norm. Write
\begin{align}
	\operatorname{tr} \bbA(z) \bbT \bbD^{-1} \bbM=A_1(z)+A_2(z)+A_3(z),
\end{align}
where
\begin{align*}
	& A_1(z)=\operatorname{tr} \sum_{j=1}^n(z \bbI-b_n \bbT)^{-1} \bbr_j \bbr_j^* \bbD_j^{-1} \bbT(\bbD^{-1}-\bbD_j^{-1}) \bbM \\
	& A_2(z)=\operatorname{tr} \sum_{j=1}^n(z \bbI-b_n \bbT)^{-1}(\bbr_j \bbr_j^* \bbD_j^{-1} \bbT \bbD_j^{-1}-n^{-1} \bbT \bbD_j^{-1} \bbT \bbD_j^{-1}) \bbM
\end{align*}
and
\begin{align*}
	A_3(z)=\operatorname{tr} \sum_{j=1}^n(z \bbI-b_n \bbT)^{-1} n^{-1} \bbT \bbD_j^{-1} \bbT(\bbD_j^{-1}-\bbD^{-1}) \bbM.
\end{align*}

We have $\mathbb{E}A_2(z)=0$ and similarly to \eqref{tr C(z)M} we have
\begin{align}
	\lvert \mathbb{E}n^{-1}A_3(z)\rvert\leq Kn^{-1}.
\end{align}
Using Lemma \ref{a(v)_quadratic_minus_trace}, \eqref{D^{-1}-D_j^{-1}} and \eqref{E|b_n-beta_1|} we get
\begin{align*}
	\mathbb{E} n^{-1} A_1(z) & =-\mathbb{E} \beta_1 \bbr_1^* \bbD_1^{-1} \bbT \bbD_1^{-1} \bbr_1 \bbr_1^* \bbD_1^{-1} \bbM(z \bbI-b_n \bbT)^{-1} \bbr_1 \\
	& =-b_n \mathbb{E}(n^{-1} \operatorname{tr} \bbD_1^{-1} \bbT \bbD_1^{-1} \bbT)(n^{-1} \operatorname{tr} \bbD_1^{-1} \bbM(z \bbI-b_n \bbT)^{-1} \bbT)+\bigo(n^{-1}) \\
	& =-b_n \mathbb{E}(n^{-1} \operatorname{tr} \bbD^{-1} \bbT \bbD^{-1} \bbT)(n^{-1} \operatorname{tr} \bbD^{-1} \bbM(z \bbI-b_n \bbT)^{-1} \bbT)+\bigo(n^{-1}) . 
\end{align*}
Using Lemmas \ref{bound moment of b_j and beta_j} and \ref{E|trTDj-1-EtrTDj-1|} we find
\begin{align*}
	& \lvert \operatorname{Cov}  (n^{-1} \operatorname{tr} \bbD^{-1} \bbT \bbD^{-1} \bbT, n^{-1} \operatorname{tr} \bbD^{-1} \bbM (z \bbI-b_n \bbT )^{-1} \bbT ) \rvert \\
	\leq &  (\mathbb{E}\lvert n^{-1} \operatorname{tr} \bbD^{-1} \bbT \bbD^{-1} \bbT\rvert^2 )^{1 / 2} n^{-1} \\
	\quad & \times(\mathbb{E}\lvert\operatorname{tr} \bbD^{-1} \bbM(z \bbI-b_n \bbT)^{-1} \bbT-\mathbb{E}\operatorname{tr} \bbD^{-1} \bbM(z \bbI-b_n \bbT)^{-1} \bbT\rvert^2)^{1 / 2} \\
	\leq & K n^{-1}.
\end{align*}
Therefore 
\begin{align}\label{n^-1A_1(z)}
	&\mathbb{E} n^{-1}  A_1(z) \\
	\nonumber= & -b_n(\mathbb{E} n^{-1} \operatorname{tr} \bbD^{-1} \bbT \bbD^{-1} \bbT) 	(\mathbb{E} n^{-1} \operatorname{tr} \bbD^{-1} \bbM(z \bbI-b_n \bbT)^{-1} \bbT)+\bigo(n^{-1}) .
\end{align}
Since $\mathbb{E}\beta_1=-z\mathbb{E}\underline{s}_n$, $\mathbb{E}\beta_1=b_n+\bigo(n^{-1})$ and Lemma \ref{|Es_n-s_n^0|}, we have $b_n=-z\underline{s}_n^0(z)+\bigo(n^{-1})$. From \eqref{D^-1(z)}, \eqref{tr B(z)M} and \eqref{tr C(z)M} we get
\begin{align}\label{En^-1trD^-1T(zI-b_nT)^-1T}
	& \mathbb{E} n^{-1}  \operatorname{tr} \bbD^{-1} \bbT(z \bbI-b_n \bbT)^{-1} \bbT \\
	\nonumber=& n^{-1} \operatorname{tr}[-(z \bbI-b_n \bbT)^{-1}+\mathbb{E} \bbB(z)+\mathbb{E} \bbC(z)] \bbT(z \bbI-b_n \bbT)^{-1} \bbT \\
	\nonumber=&  -\frac{y_n}{z^2} \int \frac{t^2 d H_p(t)}{(1+t \underline{s}_n^0)^2}+\bigo(n^{-1}) .
\end{align}
Similarly,
\begin{align}\label{En^-1trD^-1(Es_nT+I)^-1T(zI-b_nT)^-1T}
	& \mathbb{E} n^{-1}  \operatorname{tr} \bbD^{-1}(\underline{s}_n^0 \bbT+\bbI)^{-1} \bbT(z \bbI-b_n \bbT)^{-1} \bbT =-\frac{y_n}{z^2} \int \frac{t^2 d H_p(t)}{\left(1+t \underline{s}_n^0\right)^3}+\bigo(n^{-1}) .
\end{align}
Using \eqref{D^-1(z)} and \eqref{tr B(z)M}--\eqref{En^-1trD^-1T(zI-b_nT)^-1T}, we get
\begin{align*}
	&\mathbb{E} n^{-1} \operatorname{tr}  \bbD^{-1} \bbT \bbD^{-1} \bbT \\
	=&  -\mathbb{E} n^{-1} \operatorname{tr} \bbD^{-1} \bbT(z \bbI-b_n \bbT)^{-1} \bbT \\
	& \quad-b_n^2(\mathbb{E} n^{-1} \operatorname{tr} \bbD^{-1} \bbT \bbD^{-1} \bbT)(\mathbb{E} n^{-1} \operatorname{tr} \bbD^{-1} \bbT(z \bbI-b_n \bbT)^{-1} \bbT)+\bigo(n^{-1}) \\
	=&  \frac{y_n}{z^2} \int \frac{t^2 d H_p(t)}{(1+t \underline{s}_n^0)^2}\left(1+z^2 (\underline{s}_n^0)^2 \mathbb{E} n^{-1} \operatorname{tr} \bbD^{-1} \bbT \bbD^{-1} \bbT\right)+\bigo(n^{-1}).
\end{align*}
Therefore
\begin{align}\label{En^-1trD^-1TD^-1T}
	& \mathbb{E} n^{-1} \operatorname{tr} \bbD^{-1} \bbT \bbD^{-1} \bbT \\
	\nonumber& =\left[\frac{y_n}{z^2} \int \frac{t^2 d H_p(t)}{(1+t \underline{s}_n^0)^2}\right]\left[1-y_n \int \frac{(\underline{s}_n^0)^2 t^2 d H_p(t)}{(1+t \underline{s}_n^0)^2}\right]^{-1}+\bigo(n^{-1}).
\end{align}
Finally we have from \eqref{D^-1(z)}--\eqref{n^-1A_1(z)}, \eqref{En^-1trD^-1(Es_nT+I)^-1T(zI-b_nT)^-1T} and \eqref{En^-1trD^-1TD^-1T}
\begin{align*}
	& n^{-1} \mathbb{E} \operatorname{tr} \bbD^{-1}(\underline{s}_n^0 \bbT+\bbI)^{-1} \bbT \bbD^{-1} \bbT \\
	=&  -\mathbb{E} n^{-1}\operatorname{tr} \bbD^{-1}(\underline{s}_n^0 \bbT+\bbI)^{-1} \bbT(z \bbI-b_n \bbT)^{-1} \bbT \\
	&\quad -b_n^2(\mathbb{E} n^{-1} \operatorname{tr} \bbD^{-1} \bbT \bbD^{-1} \bbT) \\
	&\quad\quad \times(\mathbb{E} n^{-1} \operatorname{tr} \bbD^{-1}(\underline{s}_n^0 \bbT+\bbI)^{-1} \bbT(z \bbI-b_n \bbT)^{-1} \bbT)+O(n^{-1}) \\
	&= \frac{y_n}{z^2} \int \frac{t^2 d H_p(t)}{(1+t \underline{s}_n^0)^3} \\
	& \quad\times\left(1+z^2 (\underline{s}_n^0)^2\left[\frac{y_n}{z^2} \int \frac{t^2 d H_p(t)}{(1+t \underline{s}_n^0(z))^2}\right]\left[1-y_n \int \frac{(\underline{s}_n^0)^2 t^2 d H_p(t)}{(1+t \underline{s}_n^0)^2}\right]^{-1}\right)+\bigo(n^{-1}) \\
	=&  \left[\frac{y_n}{z^2} \int \frac{t^2 d H_p(t)}{(1+t \underline{s}_n^0)^3}\right]\left[1-y_n \int \frac{(\underline{s}_n^0)^2 t^2 d H_p(t)}{(1+t \underline{s}_n^0)^2}\right]^{-1}+\bigo(n^{-1}) . 
\end{align*}
Therefore, from \eqref{M_n^2} we conclude that under Assumption \ref{assumRG},
\begin{align}\label{sup M_n^2 in RG case}
	&\sup_{z\in\mathcal{C}}\left| M_n^2(z) -\beta_x \frac{(\underline{s}_n^0(z))^3 n^{-1}  \mtr ((\underline{s}_n^0(z)\bbT+\bbI)^{-2}\bbT)\circ ((\underline{s}_n^0(z)\bbT+\bbI)^{-1}\bbT)}{1-y_n\int \underline{s}_n^0(z)^2t^2(1+t\underline{s}_n^0(z))^{-2}dH_p(t))^2} \right.\\
	 \nonumber &\quad\quad\quad  \left. -\frac{y_n\int \underline{s}_n^0(z)^3t^2(1+t\underline{s}_n^0(z))^{-3}dH_p(t)}{(1-y_n\int \underline{s}_n^0(z)^2t^2(1+t\underline{s}_n^0(z))^{-2}dH_p(t))^2}\right| =\bigo(n^{-1});
\end{align}
under Assumption \ref{assumCG},
\begin{align}\label{sup M_n^2 in CG case}
	&\sup_{z\in\mathcal{C}}\left| M_n^2(z) - \beta_x \frac{ (\underline{s}_n^0(z))^3 n^{-1}  \mtr ((\underline{s}_n^0(z)\bbT+\bbI)^{-2}\bbT)\circ ((\underline{s}_n^0(z)\bbT+\bbI)^{-1}\bbT)}{1-y_n\int \underline{s}_n^0(z)^2t^2(1+t\underline{s}_n^0(z))^{-2}dH_p(t))^2}\right| \\
	\nonumber &=\bigo(n^{-1}).
\end{align}

Due to  $y_n\to y$, $\underline{s}_n^0(z)\to \underline{s}^0(z)$, $H_p\stackrel{d}{\rightarrow} H$, $f_m(z)\to f(z)$,
$$\frac{1}{p}\sum_{i=1}^{p}\bbe_i^*(\underline{s}^0(z)\bbT+\bbI)^{-2}\bbT\bbe_i\bbe_i^*(\underline{s}^0(z)\bbT+\bbI)^{-1}\bbT\bbe_i\to h_2(z),$$
 we have, under Assumption \ref{assumRG},
\begin{align*}
	&p\int f_m(x)d[\mathbb{E}F^{B_n}(x)-F^{y_n,H_p}(x)]=-\frac{1}{2\pi i}\oint_{\mathcal{C}}f_m(z)M_n^2(z)dz \\
	&\quad \to-\frac{1}{2\pi i}\oint_{\calC} f(z)  \frac{y\int \underline{s}^0(z)^3t^2(1+t\underline{s}^0(z))^{-3}d H(t)}{(1-y\int \underline{s}^0(z)^2t^2(1+t\underline{s}^0(z))^{-2} d H(t))^2}  dz \\
	& \quad\quad\quad  - \frac{\beta_x}{2\pi i}\oint_{\calC} f(z) \frac{y (\underline{s}^0(z))^3 h_2(z)}{1-y\int \underline{s}^0(z)^2t^2(1+t\underline{s}^0(z))^{-2} d H(t)} d z;
\end{align*}
and under Assumption \ref{assumCG},
\begin{align*}
	&p\int f_m(x)d[\mathbb{E}F^{B_n}(x)-F^{y_n,H_p}(x)]=-\frac{1}{2\pi i}\oint_{\mathcal{C}}f_m(z)M_n^2(z)dz \\
	&\quad \to -\frac{\beta_x}{2\pi i}\oint_{\calC} f(z) \frac{y \underline{s}^0(z)^3 h_2(z)}{1-y\int \underline{s}^0(z)^2t^2(1+t\underline{s}^0(z))^{-2} d H(t)} d z.
\end{align*}

\section{Proof of Theorem \ref{RateFunctionCLT}}\label{ProfRateCLT}
 The proof of Theorem \ref{RateFunctionCLT} consists of two main steps.
 First, we estimate the order of the Bernstein remainder term which are the terms $\Delta_2$ and $\Delta_3$ in \eqref{aboutdelta}; then, by applying Lemma \ref{Chen} and the results in \cite{CuiH25R}, we derive the conclusions of Theorem \ref{RateFunctionCLT}. 
 
  Taking $m=[n^{8/5-\kappa}]$. Based on Lemma \ref{convergence_rate_ESD},
 	\begin{align*}
 		\sup_{x}|F^{\bbB_n}(x)-F^{y_n,H_n}(x)|=\bigo_{a.s.}(n^{-2/5+\kappa}),
 	\end{align*}
 	we have 
 	\begin{align*}
 		&\bE |\Delta_2|=\bigo(n^{-1+2\kappa}), \\
 		&\bE |\Delta_3|=\bigo(n^{-13/5+3\kappa}).
 	\end{align*}
 	Recall the notation "$\stackrel{\epsilon_n}{\sim}$" in Remark \ref{ZsimX}. Further, we get
 	\begin{align*}
 		\int f(x)dG_n(x) \stackrel{n^{-1/2+\kappa}}{\sim} \int f_m(x)dG_n(x).
 	\end{align*}
  Notice that $f_m$ is an analytic function. By invoking the main result in \cite{CuiH25R}, under the assumptions of Theorem \ref{RateFunctionCLT} $(\romannumeral1)$ or $(\romannumeral2)$, we obtain 
  	\begin{align}\label{fmrate}
  		\mathbb{K}\left(\frac{\int f_m(x)dG_n(x)-\mu_n(f_m)}{\sqrt{\sigma_n^2(f_m)}},Z\right)\leq Kn^{-1/2+\kappa}.
  	\end{align}
  	One additional point to note regarding the above conclusions is that, due to the relaxation of the assumption that the fourth-order moment match Gaussian, the forms of the mean term $\mu_n(f_m)$ and variance term $\sigma_n^2(f_m)$ have changed. However, the primary reason this change does not affect the rates is evident from \eqref{RGinvar}, \eqref{CGinvar}, \eqref{sup M_n^2 in RG case}, \eqref{sup M_n^2 in CG case}, by using Lemma \ref{Chen}.
	
	According to \eqref{fm-f} and Assumption \ref{assumtestf},
	\begin{align*}
		\left| f_m(x)-f(x) \right|\leq \frac{K}{m}=\bigo(n^{-8/5+\kappa}).
	\end{align*}
	Thus, when $\Xi_n^c$ defined in \eqref{Xi_n} occurs,
	\begin{align*}
		& \int f(x) dG_n(x) \stackrel{n^{-1/2}}{\sim} \int f_m(x)dG_n(x), \\
		& \left|\mu_n(f_m)-\mu_n(f)\right|=o(n^{-1}), \\
		& \left|\sigma_n^2(f_m)-\sigma_n^2(f)\right|=o(n^{-2}).
	\end{align*}
	Thanks to Lemma \ref{Chen}, combing \eqref{fmrate}, we have under the assumptions of Theorem \ref{RateFunctionCLT} $(\romannumeral1)$ or $(\romannumeral2)$,
	\begin{align*}
		\mathbb{K}\left(\frac{\int f(x)dG_n(x)-\mu_n(f)}{\sqrt{\sigma_n^2(f)}},Z\right)\leq Kn^{-1/2+\kappa}.	
	\end{align*}
	The proof of Theorem \ref{RateFunctionCLT} is complete.

 \section{Useful lemmas}\label{uselema}
\begin{lemma}[Theorem 35.12 of \cite{Billingsley95}]\label{yangclt}
 Suppose that  $Y_{1}, Y_{2}, \dots, Y_{r_n}$ is a real martingale difference sequence with respect to the increasing $\sigma$-field $\left\{\mathcal{F}_{j}\right\}$ having second moments. If as $n \rightarrow \infty$,
\begin{itemize}
\item[(i)]
$ \sum_{j=1}^{r_n} \mathrm{E}\left(Y_{j}^2 \mid \mathcal{F}_{j-1}\right) \xrightarrow{p} \sigma^2,$
 where $\sigma^2$ is a positive constant, 
 \item[(ii)] and for each $\varepsilon>0$,
$ \sum_{j=1}^{r_n} \mathrm{E}\left(Y_{j}^2 I_{\left(\left|Y_{j}\right| \geq \varepsilon\right)}\right) \rightarrow 0$,
\end{itemize}
 then
 $$
 \sum_{j=1}^{r_n} Y_{j} \xrightarrow{d} N\left(0, \sigma^2\right) ,
 $$
 where "$\xrightarrow{p}$" represents convergence in probability.
\end{lemma}
  \begin{lemma}[\cite{BaiH12C}]\label{convergence_rate_ESD}
  Assume that $0<\theta\leq y_n\leq \Theta<1$ for positive constants $\theta$ and $\Theta$, $\sup_n\sup_{i,j}\bE\lvert x_{ij}\rvert^8<\infty$, $\lim_{n\to\infty}\lambda_1^{\bbT_n}=\lambda_0>0$, $\sup_n \lambda_p^{\bbT_n}<\infty$ and $F^{\bbT_n}\stackrel{d}{\to} H$, a proper distribution function. Then we have, for any $\kappa>0$,
  \begin{align*}
  	& \sup_{x}|\bE F^{\bbB_n}(x)-F^{y_n,H_n}(x)|=\bigo(n^{-1/2}), \\
  	& \sup_{x}|F^{\bbB_n}(x)-F^{y_n,H_n}(x)|=\bigo_{p}(n^{-2/5}), \\
  	& \sup_{x}|F^{\bbB_n}(x)-F^{y_n,H_n}(x)|=\bigo_{a.s.}(n^{-2/5+\kappa}).
  \end{align*}

  \end{lemma}
  


	\begin{lemma}[Lemma 9 in \cite{Chen81B}]\label{Chen}
	Let $Z_n=X_n+Y_n$, $n=1,2,\dots$, be a sequence of random variables, the distribution functions                                                                                                                                                                                                                                                                                                                                                                                                                                                                                                                                                                                                                                                                                                                                                                                                                                                                                                                                                                                                                                                                                                                                                                                                                                                                                                        of $Z_n$, $X_n$ be $F_n(x)$ and $G_n(x)$, respectively. Suppose there exists a sequence $\epsilon_n$ such that
	$$\mathbb{K}(X_n,Z)= \sup_{x\in\mathbb{R}}\lvert G_n(x)-\Phi(x)\rvert\leq K\epsilon_n,\quad \mathbb{P}\left(\lvert Y_n\rvert\geq K\epsilon_n\right)\leq K\epsilon_n, \quad n=1,2,\dots,$$
	then 
	$$\mathbb{K}(Z_n,Z)=\sup_{x\in\mathbb{R}}\lvert F_n(x)-\Phi(x)\rvert\leq K\epsilon_n,\quad n=1,2,\dots.$$ 
\end{lemma}
\begin{proof}
	On one hand, we write
	\begin{align*}
		F_n(x) &= \mathbb{P}(Z_n\leq x)=\mathbb{P}(X_n+Y_n\leq x) \\
		&\leq \mathbb{P}(X_n+Y_n\leq x,\lvert Y_n\rvert \textless K\epsilon_n)+\mathbb{P}(\lvert Y_n\rvert\geq K\epsilon_n) \\
		&\leq \mathbb{P}(X_n\leq x+K\epsilon_n) + \mathbb{P}(\lvert Y_n\rvert\geq K\epsilon_n).
	\end{align*}
	On the other hand,
	\begin{align*}
		F_n(x) &= \mathbb{P}(X_n+Y_n\leq x) \\
		&\geq \mathbb{P}(X_n+Y_n\leq x,\lvert Y_n\rvert\textless K\epsilon_n)-\mathbb{P}(\lvert Y_n\rvert\geq K\epsilon_n) \\
		&\geq \mathbb{P}(X_n\leq x-K\epsilon_n)-\mathbb{P}(\lvert Y_n\rvert\geq K\epsilon_n).
	\end{align*}
	Noting that $\Phi(x)$ is a Lipschitz continuous function, we can get
	\begin{align*}
		F_n(x)-\Phi(x) &\leq \mathbb{P}(X_n\leq x+K\epsilon_n)-\Phi(x+K\epsilon_n)+\Phi(x+K\epsilon_n)-\Phi(x) \\
		&\quad + \mathbb{P}(\lvert Y_n\rvert\geq K\epsilon_n) \\
		&= G_n(x+K\epsilon_n)-\Phi(x+K\epsilon_n)+\Phi(x+K\epsilon_n)-\Phi(x) \\
		&\quad + \mathbb{P}(\lvert Z_n-X_n\rvert\geq K\epsilon_n) \\
		&\leq K\epsilon_n,
	\end{align*}
	and
	\begin{align*}
		F_n(x)-\Phi(x) &\geq \mathbb{P}(X_n\leq x-K\epsilon_n)-\Phi(x-K\epsilon_n)+\Phi(x-K\epsilon_n)-\Phi(x) \\
		&\quad-\mathbb{P}(\lvert Y_n\rvert\geq K\epsilon_n) \\
		&= G_n(x-2K\epsilon_n)-\Phi(x-2K\epsilon_n)+\Phi(x-2K\epsilon_n)-\Phi(x) \\
		&\quad -\mathbb{P}(\lvert Z_n-X_n\rvert\geq K\epsilon_n)  \\
		&\geq -K\epsilon_n.
	\end{align*}
	Since the above two bounds are independent of $x$, we complete the proof.
\end{proof}

\begin{remark}\label{K-S_dist_equi}
	By Markov's inequality, for any fixed $s>0$, if $\bE\lvert Z_n-X_n \rvert^s\leq K\epsilon_n^{1+s}$, then 
	\begin{align}\label{ZsimX}
		\mathbb{P}\left(\lvert Z_n-X_n\rvert\geq K\epsilon_n\right)\leq K\epsilon_n.
	\end{align}
	In the sequel, we denote $Z_n\stackrel{\epsilon_n}{\sim}X_n$  if \eqref{ZsimX} holds.
\end{remark}

	\begin{lemma}[\cite{Burkholder73D}]\label{Burkholder_1}
		Let $\left\lbrace X_k\right\rbrace $ be a complex martingale difference sequence with respect to the increasing $\sigma$-field $\left\lbrace \mathcal{F}_k\right\rbrace $. Then, for $p>1$,
		$$\mathbb{E}\left| \sum X_k\right|^p\leq K_p\mathbb{E}\left(\sum \lvert X_k\rvert^2\right)^{p/2}.$$
	\end{lemma}
	\begin{lemma}\label{Burkholder_2}
		Let $\left\lbrace X_k\right\rbrace$ be a complex martingale difference sequence with respect to the increasing $\sigma$-field $\left\lbrace \mathcal{F}_k\right\rbrace $. Then, for $p\geq 2$,
		$$\mathbb{E}\left|\sum X_k\right|^p\leq K_p\left(\left(\sum \mathbb{E}\lvert X_k\rvert^2\right)^{p/2}+\mathbb{E}\sum \lvert X_k\rvert^p\right).$$
	\end{lemma}

	Recalling the definitions stated in the initial of Section \ref{zanting}, the following equations hold.
	\begin{proposition}
		For any Hermitian matrix $\bbC$ ,
		 \begin{align}\label{r_i^*(C+r_ir_i^*)^-1}
 	\bbr_i^{*}(\bbC+\bbr_i\bbr_i^{*})^{-1}=\frac{1}{1+\bbr_i^{*}\bbC^{-1}\bbr_i}\bbr_i^{*}\bbC^{-1}.
 \end{align}
 It follows that 
\begin{align} \label{D^{-1}-D_j^{-1}}
	&  \bbD^{-1}(z)-\bbD_j^{-1}(z)  =  -\frac{\bbD_j^{-1}(z)\bbr_j\bbr_j^{*}\bbD_j^{-1}(z)}{1+\bbr_j^{*}\bbD_j^{-1}(z)\bbr_j} =-\beta_{j}(z)\bbD_j^{-1}(z)\bbr_j\bbr_j^{*}\bbD_j^{-1}(z),
 \end{align}
 and
 \begin{align} \label{D_j^{-1}-D_kj^{-1}}
 	& \bbD_j^{-1}(z)-\bbD_{kj}^{-1}(z)=-\frac{\bbD_{kj}^{-1}(z)\bbr_k\bbr_k^{*}\bbD_{kj}^{-1}(z)}{1+\bbr_k^{*}\bbD_{kj}^{-1}(z)\bbr_k}=-\beta_{kj}(z)\bbD_{kj}^{-1}(z)\bbr_k\bbr_k^{*}\bbD_{kj}^{-1}(z).
 \end{align}
 In addition, 
 \begin{align}\label{beta_1}
 	\beta_j=b_n-\beta_j b_n \gamma_j=b_n-b_n^2\gamma_j+\beta_j b_n^2\gamma_j^2 ,
 \end{align}
 \begin{align}\label{beta_ij decomposition}
 	\beta_{ij}=b_{ij}-\beta_{ij}b_{ij}\gamma_{ij}=b_{ij}-b_{ij}^2\gamma_{ij}+\beta_{ij} b_{ij}^2\gamma_{ij}^2 ,
 \end{align}
 and
 \begin{align}\label{beta_j}
 	\beta_{j}(z)=\widetilde{\beta}_j(z)-\beta_{j}(z)\widetilde{\beta}_j(z)\varepsilon_{j}(z).
 \end{align}

	\end{proposition}

	\begin{lemma}\label{bound moment of b_j and beta_j}
		 When $\Xi_n^c$ (defined in \eqref{Xi_n}) occurs, For all $z\in\mathcal{C}$, any $p>0$, $i, j=1,\dots,n$, we have
		\begin{align*}
			& \max\left(\lVert\bbD^{-1}(z)\rVert^p,\lVert \bbD_j^{-1}(z)\rVert^p,\lVert \bbD_{ij}^{-1}(z)\rVert^p\right)\leq K_p, \\
			& \max\left(\lvert \beta_j(z)\rvert^p, \lvert b_j(z)\rvert^p, \lvert \widetilde{\beta}_j(z)\rvert^p\right) \leq K_p.
		\end{align*}
	\end{lemma}
	
	\begin{proof}
		For all $z\in \calC_u $ or $\overline{\calC_u}$, according to (3.4) in \cite{BaiS98N}, we have 
		\begin{align*}
			& \max\left(\lVert\bbD^{-1}(z)\rVert^p,\lVert \bbD_j^{-1}(z)\rVert^p,\lVert \bbD_{ij}^{-1}(z)\rVert^p\right)\leq \frac{1}{(\Im z)^p}, \\
			& \max\left( \lvert \beta_j(z)\rvert^p, \lvert b_j(z)\rvert^p, \lvert  \widetilde{\beta}_j(z)\rvert^p\right) \leq \frac{|z|^p}{(\Im z)^p}.
		\end{align*}
		As for $z\in\calC_r$ or $\overline{\calC_r}$, due to $\Xi_n^c$ occuring, there are non-zero gaps between the eigenvalues of $\bbB_n$ and $x_r$. Then,
		\begin{align}\label{zinCr}
			\max\left(\lVert\bbD^{-1}(z)\rVert^p,\lVert \bbD_j^{-1}(z)\rVert^p,\lVert \bbD_{ij}^{-1}(z)\rVert^p\right)\leq \frac{1}{(x_r-\lambda_{max})^p}\leq \frac{2^p}{\epsilon^p}.
		\end{align}
		Next, from \eqref{zinCr} and $\Xi_n^c$ occuring, we have
		\begin{align*}
			\left|1+\bbr_j^*\bbD_j^{-1}(z)\bbr_j\right|\leq 1+\left|\bbr_j^*\bbD_j^{-1}(z)\bbr_j \right| \leq 1+n^{-1}\lVert \bbD_j^{-1}(z) \rVert \operatorname{tr}\bbB_j \leq 1+K_p(x_r-\epsilon/2).
		\end{align*}
		Further,
		\begin{align*}
			|\beta_j(z)|^p\leq \frac{|1+\bbr_j^*\bbD_j^{-1}(z)\bbr_j|}{|1+\bbr_j^*\bbD_j^{-1}(z)\bbr_j|^{p+1}} \leq K_p(1+K_p(x_r-\epsilon/2)).
		\end{align*}
		Following a similar argument, we can get for all $z\in\calC_r$ or $\overline{\calC_r}$,
		\begin{align*}
			\max\left(\lvert b_j(z)\rvert^p,\lvert \widetilde{\beta}_j(z)\rvert^p\right)\leq K_p.
		\end{align*}
		
		As for $z\in\calC_l$ or $\overline{\calC_l}$, due to $\Xi_n^c$ occuring, there are non-zero gaps between the eigenvalues of $\bbB_n$ and $x_l$. Thus, the above results also true by replacing $x_r$ with $x_l$.		
	\end{proof}
	
	\begin{lemma}[Lemma C.5 in \cite{CuiH25R}]\label{bound spectral norm of zI-bj(z)T}
		For all $z\in\mathcal{C}$,
		\begin{align*}
			\lVert (z\bbI-\frac{n-1}{n}b_j(z)\bbT)^{-1}\rVert\leq K.
		\end{align*}
	\end{lemma}
	
	\begin{lemma}[Lemma 2.7 in \cite{BaiS98N}]\label{QMTr}
		For $\bbX=(x_1,\dots,x_n)^{T}$ i.i.d. standardized (complex) entries, $\bbC$ $ n\times n$ matrix (complex),  we have, for any $p\geq2$,
		\begin{align*}
			\mathbb{E}\lvert \bbX^{*}\bbC\bbX-\operatorname{tr}\bbC\rvert^p\leq K_p((\mathbb{E}\lvert x_1\rvert^4\operatorname{tr}\bbC\bbC^{*})^{p/2}+\mathbb{E}\lvert x_1\rvert^{2p}\operatorname{tr}(\bbC\bbC^{*})^{p/2}).
		\end{align*}
	\end{lemma}
	\begin{lemma}\label{QMTr_1}
		For nonrandom Hermitian nonnegative definite $p\times p$ matrices $\bbA_l$, $l=1,\dots,k$,
		\begin{align}\label{QMTr_1result}
			\mathbb{E}\left| \prod_{l=1}^{k}(\bbr_1^{*}\bbA_l\bbr_1-n^{-1}\operatorname{tr}\bbT\bbA_l)\right|\leq K_k n^{-k/2}  \prod_{l=1}^{k} \lVert \bbA_l \rVert.
		\end{align}
	\end{lemma}
	\begin{proof}
		Recalling the truncation procedure from Section \ref{trunnorm},  $\mathbb{E}\lvert X_{11}\rvert^{8}<\infty$, and Lemma \ref{QMTr}, we have, for all $p>1$,
		\begin{align}\label{QMTr_2}
			\mathbb{E} \lvert \bbr_1^{*}\bbA_l \bbr_1-n^{-1}\operatorname{tr}\bbT\bbA_l \rvert^p &\leq K_p n^{-p} [ (\mathbb{E}\lvert x_{11}\rvert^4 \operatorname{tr}\bbA_l\bbA_l^{*})^{p/2}+\mathbb{E}\lvert x_{11}\rvert^{2p}\operatorname{tr}(\bbA_l\bbA_l^{*})^{p/2}] \\
			\nonumber &\leq K_p n^{-p} \lVert \bbA_l\rVert^p \left[ n^{p/2} + (n^{1/4} \eta_n)^{(2p-8)\vee 0}n\right]  \\
			\nonumber &= K_p \lVert \bbA_l\rVert^p n^{-p/2}.
		\end{align}
		Then , (\ref{QMTr_1result}) can get from (\ref{QMTr_2}) and H\"older's inequality.
	\end{proof}
	
	\begin{remark}
		Notice that, while $\bbA_l$ is stated as nonrandom in lemma \ref{QMTr_1}, the result extends to cases where $\bbA_l$ is random, provided that it is independent of $\bbr_1$. In such instances, the right-hand side of (\ref{QMTr_1result}) involves 
		$$K_k n^{-k/2} \mathbb{E} \prod_{l=1}^{k} \lVert \bbA_l \rVert .$$
	\end{remark}
	
	From the above lemma, we can further get the following lemmas.
	\begin{lemma}\label{Q_QMTr}
		For nonrandom Hermitian nonnegative definite $p\times p$ $\bbA_r$, $r=1,\dots,s$ and  $\bbB_l$, $l=1,\dots,k$, we  establish the following inequality:
		\begin{align*}
			&  \left| \mathbb{E} \left( \prod_{r=1}^{s}\bbr_1^{*}\bbA_r \bbr_1\prod_{l=1}^{k}(\bbr_1^{*}\bbB_l\bbr_1-n^{-1}\operatorname{tr}\bbT\bbB_l)\right) \right| \leq K n^{-k/2} \prod_{r=1}^{s}\lVert \bbA_r \rVert \prod_{l=1}^{k} \lVert \bbB_l \rVert, \quad s\geq0,k\geq0.
		\end{align*}
	\end{lemma}
	
	According to (3.2) in \cite{BaiS04C}, Lemma \ref{Q_QMTr} can be extended.
	\begin{lemma}\label{a(v)_quadratic_minus_trace}
		For $a(v)$ and $p\times p$ matrices $\bbB_l(v)$, $l=1,\dots,k$, which are independent of $\bbr_1$, both satisfy
		$$max(\lvert a(v)\rvert,\lVert \bbB_l(v)\rVert)\leq K\left(1+n^s I\left\lbrace \lVert \bbB_n\rVert\geq x_r ~\mbox{or}~ \lambda_{min}^{\widetilde{\bbB}}\leq x_l                                                                                \right\rbrace \right)$$
		for some positive $s$, with $\widetilde{\bbB}$ being $\bbB_n$ or $\bbB_n$ with one or two of $\bbr_j$'s removed. We get
		\begin{align*}
			\left|\mathbb{E}\left(a(v)\prod_{l=1}^{k}(\bbr_1^{*}\bbB_l(v)\bbr_1-n^{-1}\operatorname{tr}\bbT\bbB_l(v))\right)\right|\leq K n^{-k/2}.
		\end{align*}
	\end{lemma}
	
	\begin{lemma}\label{quadratic form minus trace of qudratic form}
		For nonrandom Hermitian nonnegative definite $p\times p$ $\bbD$ with bounded spectral norm,
		\begin{align*}
			\mathbb{E}\lvert \bbr_1^{*}\bbD\bbr_2\bbr_2^{*}\bbD\bbr_1-n^{-1}\operatorname{tr}\bbT\bbD\bbr_2\bbr_2^{*}\bbD\rvert^{p}\leq Kn^{-p+((p/2-2)\vee 0)}\eta_n^{(2p-8)\vee 0}.
		\end{align*}
	\end{lemma}
	\begin{proof}
		Recalling the truncation steps in Section \ref{trun},  $\mathbb{E}\lvert x_{11}\rvert^{8}<\infty$, Lemma \ref{QMTr} and \ref{Q_QMTr}, we have, for all $p>1$,
		\begin{align*}
			& \mathbb{E}\lvert \bbr_1^{*}\bbD\bbr_2\bbr_2^{*}\bbD\bbr_1-n^{-1}\operatorname{tr}\bbT\bbD\bbr_2\bbr_2^{*}\bbD\rvert^{p} \\
			\leq& Kn^{-p}[(\mathbb{E}\lvert x_{11}\rvert^4\mathbb{E}\operatorname{tr}\bbD\bbr_2\bbr_2^{*}\bbD\bbD\bbr_2\bbr_2^{*}\bbD)+\mathbb{E}\lvert x_{11}\rvert^{2p}\mathbb{E}tr(\bbD\bbr_2\bbr_2^{*}\bbD)^p] \\
			\leq& n^{-p}[K+K(n^{1/4}\eta_n)^{(2p-8)\vee 0}\mathbb{E}(\bbr_2^{*}\bbD^{2}\bbr_2)^p] \\
			\leq& Kn^{-p+((p/2-2)\vee 0)}\eta_n^{(2p-8)\vee 0}.
		\end{align*}
		Finally, we get the result of this lemma.
	\end{proof}
	
	The following lemma address the expectation of random quadratic forms.
	
	\begin{lemma}[(1.15) in \cite{BaiS04C}]\label{product of two quadratic minus trace of two matrices}
		For nonrandom $n\times n$ $\bbA=(a_{ij})$ and $\bbB=(b_{ij})$,
		\begin{align*}
			&  \mathbb{E}(\bbX_{\cdot 1}^{*}\bbA\bbX_{\cdot 1}-\operatorname{tr}\bbA)(\bbX_{\cdot 1}^{*}\bbB\bbX_{\cdot 1}-\operatorname{tr}\bbB)  \\
			&  \quad\quad=\beta_x \operatorname{tr}(\bbA\circ \bbB)+\alpha_x \operatorname{tr} \bbA\bbB^{T}+\operatorname{tr}\bbA\bbB
		\end{align*}
		where $\circ$ is the Hadamard product of two matrices.
	\end{lemma}

	\begin{lemma}[Lemma C.12 in \cite{CuiH25R}]\label{lem3Qua}
		For nonrandom  nonnegative $n\times n$ $\bbA=(a_{ij})$, $\bbB=(b_{ij})$ and $\bbC=(c_{ij})$, there are bounded spectral norm.   For $j=1,\dots,n$, we have
		\begin{align*}
			\lvert\mathbb{E}(\bbX_{\cdot j}^*\bbA\bbX_{\cdot j}-\operatorname{tr}\bbA)(\bbX_{\cdot j}^*\bbB\bbX_{\cdot j}-\operatorname{tr}\bbB)(\bbX_{\cdot j}^*\bbC\bbX_{\cdot j}-\operatorname{tr}\bbC)\rvert\leq Kn.
		\end{align*}
	\end{lemma}

	\begin{lemma}\label{E|trTDj-1-EtrTDj-1|}
		For all $z\in\mathcal{C}$ and any $m\geq 1$, 
		\begin{align*}
			\mathbb{E}\lvert \operatorname{tr}\bbT\bbD_j^{-1}(z)-\mathbb{E}\operatorname{tr}\bbT\bbD_j^{-1}(z)\rvert^m\leq K.
		\end{align*}
	\end{lemma}
	\begin{proof}
		Due to the martingale difference decomposition, \eqref{D_j^{-1}-D_kj^{-1}},  Lemmas \ref{Burkholder_2}, \ref{bound moment of b_j and beta_j} and \ref{a(v)_quadratic_minus_trace}, we write for any $m>1$,
		\begin{align*}
			&\mathbb{E}\lvert \operatorname{tr}\bbT\bbD_j^{-1}(z)-\mathbb{E}\operatorname{tr}\bbT\bbD_j^{-1}(z)\rvert^m \\
			=&\mathbb{E}\lvert \sum_{i=1}^{n}(\mathbb{E}_i-\mathbb{E}_{i-1})\operatorname{tr} \bbT\bbD_j^{-1}(z)\rvert^m \\
			=&\mathbb{E}\lvert \sum_{i=1}^{n}(\mathbb{E}_i-\mathbb{E}_{i-1})\operatorname{tr} \bbT(\bbD_j^{-1}(z)-\bbD_{ij}^{-1}(z))\rvert^m \\
			=&\mathbb{E}\lvert \sum_{i=1}^{n} (\mathbb{E}_i-\mathbb{E}_{i-1})\beta_{ij}(z)\bbr_i^{*}\bbD_{ij}^{-1}(z)\bbT\bbD_{ij}^{-1}(z)\bbr_i\rvert^m\\
			\leq& K \left[\mathbb{E}\sum_{i=1}^{n}\lvert (\mathbb{E}_i-\mathbb{E}_{i-1})\beta_{ij}(z)\bbr_i^{*}\bbD_{ij}^{-1}(z)\bbT\bbD_{ij}^{-1}(z)\bbr_i\rvert^m\right.  \\
			& \quad\quad\quad\quad \left.  +(\sum_{i=1}^{n}\mathbb{E}\lvert (\mathbb{E}_i-\mathbb{E}_{i-1})\beta_{ij}(z)\bbr_i^{*}\bbD_{ij}^{-1}(z)\bbT\bbD_{ij}^{-1}(z)\bbr_i\rvert^2)^{m/2}\right] \\
			\leq & K \left[  \sum_{i=1}^{n}\mathbb{E}\lvert(\mathbb{E}_i-\mathbb{E}_{i-1})[(\beta_{ij}(z)-\widetilde{\beta}_{ij}(z))\bbr_i^{*}\bbD_{ij}^{-1}(z)\bbT\bbD_{ij}^{-1}(z)\bbr_i \right. \\
			&\quad\quad\quad\quad\quad\quad\quad\quad\quad\quad +\widetilde{\beta}_{ij}(z)\bbr_i^{*}\bbD_{ij}^{-1}(z)\bbT\bbD_{ij}^{-1}(z)\bbr_i]\rvert^m  \\
			&\quad\quad +(\sum_{i=1}^{n}\mathbb{E}\lvert(\mathbb{E}_i-\mathbb{E}_{i-1}) [(\beta_{ij}(z)-\widetilde{\beta}_{ij}(z))\bbr_i^{*}\bbD_{ij}^{-1}(z)\bbT\bbD_{ij}^{-1}(z)\bbr_i \\
			&\quad\quad\quad\quad\quad\quad\quad\quad\quad\quad\quad\quad \left. +\widetilde{\beta}_{ij}(z)\bbr_i^{*}\bbD_{ij}^{-1}(z)\bbT\bbD_{ij}^{-1}(z)\bbr_i ]\rvert^2)^{m/2} \right] \\
			\leq& K\left[ \sum_{i=1}^{n}[\mathbb{E}\lvert \beta_{ij}(z)\widetilde{\beta}_{ij}(z)(\bbr_i^{*}\bbD_{ij}^{-1}(z)\bbr_i-n^{-1}\operatorname{tr}\bbT\bbD_{ij}^{-1}(z))\bbr_i^{*}\bbD_{ij}^{-1}(z)\bbT\bbD_{ij}^{-1}(z)\bbr_i\rvert^m  \right. \\
			& \quad\quad\quad \left. + \mathbb{E}\lvert \widetilde{\beta}_{ij}(z)[\bbr_i^{*}\bbD_{ij}^{-1}(z)\bbT\bbD_{ij}^{-1}(z)\bbr_i-n^{-1}\operatorname{tr} \bbT\bbD_{ij}^{-1}(z)\bbT\bbD_{ij}^{-1}(z)]\rvert^m]\right] \\
			&+K\left[\sum_{i=1}^{n} \mathbb{E}\lvert \beta_{ij}(z)\widetilde{\beta}_{ij}(z)(\bbr_i^{*}\bbD_{ij}^{-1}(z)\bbr_i-n^{-1}\operatorname{tr}\bbT\bbD_{ij}^{-1}(z))\bbr_i^{*}\bbD_{ij}^{-1}(z)\bbT\bbD_{ij}^{-1}(z)\bbr_i\rvert^2\right.\\
			&\quad\quad \left. + \sum_{i=1}^{n}\mathbb{E}\lvert \widetilde{\beta}_{ij}(z)[\bbr_i^{*}\bbD_{ij}^{-1}(z)\bbT\bbD_{ij}^{-1}(z)\bbr_i-n^{-1}\operatorname{tr} \bbT\bbD_{ij}^{-1}(z)\bbT\bbD_{ij}^{-1}(z)]\rvert^2\right]^{m/2} \\
			\leq& K.
		\end{align*}
		If $m=1$, based on the above proof process and Lemma \ref{Burkholder_1},
		\begin{align*}
			&\mathbb{E}\lvert \operatorname{tr}\bbT\bbD_j^{-1}(z)-\mathbb{E}\operatorname{tr}\bbT\bbD_j^{-1}(z)\rvert \\
			=&\mathbb{E}\lvert \sum_{i=1}^{n} (\mathbb{E}_i-\mathbb{E}_{i-1})\beta_{ij}(z)\bbr_i^{*}\bbD_{ij}^{-1}(z)\bbT\bbD_{ij}^{-1}(z)\bbr_i\rvert\\
			\leq & (\sum_{i=1}^{n}\mathbb{E}\lvert(\mathbb{E}_i-\mathbb{E}_{i-1})\beta_{ij}(z)\bbr_i^{*}\bbD_{ij}^{-1}(z)\bbT\bbD_{ij}^{-1}(z)\bbr_i\rvert^2)^{1/2} \\
			\leq & K.
		\end{align*}Then we complete the proof of this lemma.
	\end{proof}
	
\begin{lemma}\label{E|beta_j-b_j|}
		For all $z\in\mathcal{C}$ and any $m\geq 1$,
		\begin{align*}
			\mathbb{E}\lvert \widetilde{\beta}_j(z)-b_j(z)\rvert^m=Kn^{-m}.
		\end{align*}
	\end{lemma}  
	\begin{proof}
		Due to   Lemmas \ref{bound moment of b_j and beta_j}, \ref{a(v)_quadratic_minus_trace} and \ref{E|trTDj-1-EtrTDj-1|}, we write for any $m\geq1$,
		\begin{align*}
			& \mathbb{E}\lvert \widetilde{\beta}_j(z)-b_j(z)\rvert^m \\
			=& \mathbb{E}\lvert \widetilde{\beta}_j(z)b_j(z)(n^{-1}\operatorname{tr}\bbT\bbD_j^{-1}(z)-n^{-1}\mathbb{E}\operatorname{tr}\bbT\bbD_j^{-1}(z))\rvert^m \\
			\leq& K n^{-m} (\mathbb{E}\lvert \operatorname{tr}\bbT\bbD_j^{-1}(z)-\mathbb{E}\operatorname{tr}\bbT\bbD_j^{-1}(z)\rvert^{2m})^{1/2}(\mathbb{E}\lvert\widetilde{\beta}_j(z)\rvert^{2m})^{1/2} \\
			\leq & Kn^{-m}(\mathbb{E}\lvert \beta_j(z)\rvert^{2m}+\mathbb{E}\lvert \beta_j(z)\widetilde{\beta}_j(z)(\bbr_j^*\bbD_j^{-1}(z)\bbr_j-n^{-1}\operatorname{tr}\bbT\bbD_j^{-1}(z))\rvert^{2m})^{1/2}  \\
			\leq & Kn^{-m},
		\end{align*}
	which completes the proof.
	\end{proof}
	
	\begin{lemma}\label{|b_j-b|}
		For all $z\in\mathcal{C}$,
		$$\lvert b_j(z)-b(z)\rvert\leq Kn^{-1}.$$
	\end{lemma}
	
	\begin{proof}
		Due to Lemma \ref{bound moment of b_j and beta_j} and \eqref{D^{-1}-D_j^{-1}}, we have
		\begin{align}\label{b_j-b}
			\lvert b_j(z)-b(z)\rvert &=\lvert b_j(z) b(z) (n^{-1}\mathbb{E}\operatorname{tr}\bbT\bbD^{-1}(z)-n^{-1}\mathbb{E}\operatorname{tr} \bbT\bbD_j^{-1}(z))\rvert  \\
			\nonumber &\leq Kn^{-1} \lvert  \mathbb{E}\operatorname{tr}(\bbD^{-1}(z)-\bbD_j^{-1}(z))\bbT\rvert \\
			\nonumber &= Kn^{-1} \lvert \mathbb{E}\beta_{j}(z) \bbr_j^{*}\bbD_j^{-1}(z)\bbT\bbD_j^{-1}(z)\bbr_j\rvert.
		\end{align}
		Due to \eqref{beta_j} and Lemma \ref{a(v)_quadratic_minus_trace}, we can get
		\begin{align*}
			\eqref{b_j-b} &\leq Kn^{-1}\lvert \mathbb{E}\widetilde{\beta}_j(z)\bbr_j^{*}\bbD_j^{-1}(z)\bbT\bbD_j^{-1}(z)\bbr_j\rvert\\
			&\quad + Kn^{-1}\lvert\mathbb{E}\beta_{j}(z)\widetilde{\beta}_j(z)\varepsilon_j(z)\bbr_j^{*}\bbD_j^{-1}(z)\bbT\bbD_j^{-1}(z)\bbr_j\rvert \\
			&\leq Kn^{-1}\mathbb{E}\lvert\widetilde{\beta}_j(z)\bbr_j^{*}\bbD_j^{-1}(z)\bbT\bbD_j^{-1}(z)\bbr_j\rvert \\
			&\quad + Kn^{-1}\mathbb{E}\lvert\beta_{j}(z)\widetilde{\beta}_j(z)\bbr_j^{*}\bbD_j^{-1}(z)\bbT\bbD_j^{-1}(z)\bbr_j\varepsilon_j(z)\rvert \\
			&\leq Kn^{-1}+Kn^{-3/2}\leq Kn^{-1},
		\end{align*}
	which completes the proof.
	\end{proof}
	
	\begin{lemma}\label{|Es_n-s_n^0|}
		For all $z\in\mathcal{C}$, we have
		\begin{align*}
			\lvert\mathbb{E}\underline{s}_n(z)-\underline{s}_n^0(z)\rvert \leq Kn^{-1}.
		\end{align*}
	\end{lemma}
	
	\begin{proof}
		The proof of this lemma is based on arguments in Section 5 of \cite{BaiS98N}, with the extension that the result holds for all $z\in\mathcal{C}$. The uniform bound on $\lVert \bbD\rVert_1^{-1}$ provided by Lemma \ref{bound moment of b_j and beta_j}, ensures the term $\sup _{z \in\mathcal{C}} \mathbb{E}\operatorname{tr}\bbD_1^{-1} \bar{\bbD}_1^{-1} \leq p \sup_{z\in\mathcal{C}}\mathbb{E}\lVert\bbD_1^{-2}\rVert\lVert\bar{\bbD}_1^{-2}\rVert \leq K n$. Moreover, we have, for all $n$ there exists an $\varrho>0$,
		\begin{align*}
			\left(\frac{\Im\underline{s}_n^0y_n\int\frac{t^2dH_p(t)}{\lvert 1+t\underline{s}_n^0\rvert^2}}{\Im z+\Im\underline{s}_n^0y_n\int\frac{t^2dH_p(t)}{\lvert 1+t\underline{s}_n^0\rvert^2}}\right)^{1/2}\leq 1-\varrho.
		\end{align*}
		The remaining part of the proof is identical to that of  Section 5 in  \cite{BaiS98N}, and is therefore omitted.
	\end{proof}

\begin{appendix}

\section{Proof of (\ref{Qjdecom}).}\label{ProfQj}

The purpose of this section is devoted to the proof of \eqref{Qjdecom}, which is
\begin{align*}
	\sum_{j=1}^{n}\oint_{\mathcal{C}}f^{\prime}(z) (\mathbb{E}_j-\mathbb{E}_{j-1})Q_j(z)dz=o_p(1).
\end{align*}
Using the Markov's inequality and Lemma \ref{a(v)_quadratic_minus_trace},
\begin{align}\label{largeventTay}
	\bP\left(|\epsilon_j(z)\tilde{\beta}_j(z)|\geq1/2\right)\leq \frac{\bE|\epsilon_j(z)\tilde{\beta}_j(z)|^{m}}{2^{m}}=O(n^{-m/2}),
\end{align}
for any $m>0$. Thus, the event $\{|\epsilon_j(z)\tilde{\beta}_j(z)|< 1/2\}$ holds with high probability. Thus, the quantity
$
1+\varepsilon_j(z)\tilde{\beta}_j(z)
$
remains in a domain bounded away from both the negative real axis and the origin with high probability. Thus, with high probability, $\Log(1+\varepsilon_j(z)\tilde{\beta}_j(z))$ is well defined and its Taylor expansion can be applied, which is,
\begin{align}\label{Tayexp}
		\left| \Log(1+\varepsilon_j(z)\tilde{\beta}_j(z))-\varepsilon_j(z)\tilde{\beta}_j(z)\right|\leq K|\varepsilon_j(z)\tilde{\beta}_j(z)|^2.
\end{align}

From the definition of $Q_j$ provided below \eqref{decomposition of random part}, it follows that	
\begin{align}
	\nonumber &\sum_{j=1}^{n}\oint_{\mathcal{C}}f^{\prime}(z) (\mathbb{E}_j-\mathbb{E}_{j-1})Q_j(z)dz \\
	=&  \sum_{j=1}^{n}\oint_{\mathcal{C}}f^{\prime}(z) (\mathbb{E}_j-\mathbb{E}_{j-1})\varepsilon_j(z)(\widetilde{\beta}_j(z)-b_j(z\label{the first term of simplification of Qj}))dz \\
	& +\sum_{j=1}^{n}\oint_{\mathcal{C}}f^{\prime}(z) (\mathbb{E}_j-\mathbb{E}_{j-1})(\Log(1+\varepsilon_j(z)\tilde{\beta}_j(z))-\varepsilon_j(z)\tilde{\beta}_j(z)) dz  \label{the second term of simplification of Qj} 
\end{align}
We proceed to analyze items \eqref{the first term of simplification of Qj}, \eqref{the second term of simplification of Qj} individually.

\textbf{Analysis of \eqref{the second term of simplification of Qj}}. Let $R_j:=\Log(1+\varepsilon_j(z)\tilde{\beta}_j(z))-\varepsilon_j(z)\tilde{\beta}_j(z)$. For any fixed $m\geq 2$, according to Lemma \ref{Burkholder_2},
\begin{align}
	\nonumber &  \mathbb{E}\left| \sum_{j=1}^{n}\oint_{\mathcal{C}}f^{\prime}(z) (\mathbb{E}_j-\mathbb{E}_{j-1})R_j(z) dz \right| \\
	& \leq  K_m \left(\sum_{j=1}^{n}\mathbb{E}\left| \oint_{\mathcal{C}}f^{\prime}(z)(\mathbb{E}_j-\mathbb{E}_{j-1})R_j(z) dz\right|^2\right)^{m/2}.\label{the second term of Taylor expansion 2}
\end{align}
With regards to \eqref{the second term of Taylor expansion 2}, we have
\begin{align*}
	& \left(\sum_{j=1}^{n}\mathbb{E}\left| \oint_{\mathcal{C}}f^{\prime}(z)(\mathbb{E}_j-\mathbb{E}_{j-1})R_j(z) dz\right|^2 \right)^{1/2} \\
	\leq & K_m \left[ \sum_{j=1}^{n}\left[(x_l-x_r)\int_{x_l}^{x_r} \mathbb{E}\left[ \lvert R_j(u+iv_0)\rvert^2  \right] du \right. \right.   +(x_l-x_r)\int_{x_l}^{x_r} \mathbb{E}\left[\lvert R_j(u-iv_0)\rvert^2 \right] du \\
	&\quad\quad+ 2v_0\int_{-v_0}^{v_0} \mathbb{E}\left[\lvert R_j(x_r+iv)\rvert^2 \right] dv  \left. \left. + 2v_0 \int_{-v_0}^{v_0} \mathbb{E}\left[\lvert R_j(x_l+iv)\rvert^2 \right] dv  \right]\right]^{1/2}.
\end{align*}
The inequality above is due to the estimation formulas for complex integrals: for any $z=u+iv$ and integrable function $g(z)$ on $\cal{C}$,
\begin{align*}
	\left|\oint_{\calC}g(z)dz \right|\leq \oint_{\calC}|g(z)| |dz|, \quad\quad\quad \text{where}~ |dz|=\sqrt{(dx)^2+(dy)^2}.
\end{align*}
Next, according to \eqref{Tayexp} and Lemma \ref{Q_QMTr},
\begin{align}\label{R_j(z)^mI<1}
	&  \mathbb{E}\lvert R_j(z)\rvert^m   \leq K_m \mathbb{E}\lvert \varepsilon_j(z)\widetilde{\beta}_j(z)\rvert^{2m} 
	\leq K_m n^{-m}.
\end{align}
Further, the order of \eqref{the second term of Taylor expansion 2} is 
\begin{align}\label{the_second_D_term_of_lemma4.2}
	&  \left[\sum_{j=1}^{n}\mathbb{E}\left| \oint_{\mathcal{C}}f^{\prime}(z)(\mathbb{E}_j-\mathbb{E}_{j-1})R_j(z)\Theta_j(z) dz\right|^2\right]^{1/2} \leq Kn^{-1/2}.
\end{align}

\textbf{Estimation of \eqref{the first term of simplification of Qj}}. Similarly to the proof of \eqref{the second term of Taylor expansion 2},\begin{align}
	\nonumber & \mathbb{E}\left| \sum_{j=1}^{n}\oint_{\mathcal{C}}f^{\prime}(z)(\mathbb{E}_j-\mathbb{E}_{j-1})\varepsilon_j(z)(\widetilde{\beta}_j(z)-b_j(z))dz\right|^m \\
	& \leq  K_m\left(\sum_{j=1}^{n}\mathbb{E}\left| \oint_{\mathcal{C}}f^{\prime}(z)(\mathbb{E}_j-\mathbb{E}_{j-1})\varepsilon_j(z)(\widetilde{\beta}_j(z)-b_j(z))dz\right|^2\right)^{m/2}.\label{the second term of betaj-bj}
\end{align}
We only need to estimate $\mathbb{E}\lvert \varepsilon_j(z)(\widetilde{\beta}_j(z)-b_j(z))\rvert^2$, for $z\in\mathcal{C}$. According to H\"older's inequality, Lemmas \ref{QMTr_1} and \ref{E|beta_j-b_j|},
\begin{align}\label{E|varepsilon(tilde_beta_j-b_j)|m}
	&  \mathbb{E}\lvert \varepsilon_j(z)(\widetilde{\beta}_j(z)-b_j(z))\rvert^2\leq \sqrt{\mathbb{E}\lvert\varepsilon_j(z)\rvert^{4}}\sqrt{\mathbb{E}\lvert \widetilde{\beta}_j(z)-b_j(z)\rvert^{4}} \\
	\nonumber	& \leq  K\sqrt{n^{-2}}\sqrt{n^{-4}}=K n^{-3}.
\end{align}
Based on \eqref{E|varepsilon(tilde_beta_j-b_j)|m}, the estimation of \eqref{the second term of betaj-bj} is derived as
\begin{align*}
	&   \left(\sum_{j=1}^{n}\mathbb{E}\left| \oint_{\mathcal{C}}f^{\prime}(z)(\mathbb{E}_j-\mathbb{E}_{j-1})\varepsilon_j(z)(\widetilde{\beta}_j(z)-b_j(z))dz\right|^2\right)^{1/2} \leq Kn^{-1}.
\end{align*}

\end{appendix}


\begin{thebibliography}{}
\bibitem[Bai, 1999]{Bai99M}
Bai, Z. (1999).
\newblock Methodologies in spectral analysis of large dimensional random matrices, a review.
\newblock {\em Statistica Sinica}, 9(3):611--662.

\bibitem[Bai et~al., 2009]{BaiJ09C}
Bai, Z., Jiang, D., Yao, J., and Zheng, S. (2009).
\newblock Corrections to {LRT} on large-dimensional covariance matrix by {RMT}.
\newblock {\em The Annals of Statistics}, 37(6B):3822--3840.

\bibitem[Bai et~al., 2007]{BaiM07A}
Bai, Z., Miao, B., and Pan, G. (2007).
\newblock On asymptotics of eigenvectors of large sample covariance matrix.
\newblock {\em The Annals of Probability}, 35(4):1532--1572.

\bibitem[Bai and Silverstein, 1998]{BaiS98N}
Bai, Z. and Silverstein, J.~W. (1998).
\newblock No eigenvalues outside the support of the limiting spectral distribution of large-dimensional sample covariance matrices.
\newblock {\em The Annals of Probability}, 26(1):316--345.

\bibitem[Bai and Silverstein, 2004]{BaiS04C}
Bai, Z. and Silverstein, J.~W. (2004).
\newblock {CLT} for linear spectral statistics of large-dimensional sample covariance matrices.
\newblock {\em The Annals of Probability}, 32(1A):553--605.

\bibitem[Bai and Silverstein, 2010]{BaiS10S}
Bai, Z. and Silverstein, J.~W. (2010).
\newblock {\em Spectral Analysis of Large Dimensional Random Matrices}.
\newblock Springer Series in Statistics. Springer New York, New York, NY.

\bibitem[Bai et~al., 2010]{BaiW10F}
Bai, Z., Wang, X., and Zhou, W. (2010).
\newblock Functional {CLT} for sample covariance matrices.
\newblock {\em Bernoulli}, 16(4):1086--1113.

\bibitem[Bai et~al., 2012]{BaiH12C}
Bai, Z., Hu, J., and Zhou, W. (2012) 
\newblock Convergence rates to the {Marchenko-Pastur} type distribution. 
\newblock {\em Stochastic Processes and their Applications}, 122(1):68--92.

\bibitem[Bai and Yin, 1993]{BaiY93L}
Bai, Z. and Yin, Y. (1993).
\newblock Limit of the smallest eigenvalue of a large dimensional sample covariance matrix.
\newblock {\em The Annals of Probability}, 21(3):1275--1294.

\bibitem[Bao et~al., 2015]{BaoL15S}
Bao, Z., Lin, L.-C., Pan, G., and Zhou, W. (2015).
\newblock Spectral statistics of large dimensional spearman's rank correlation
matrix and its application.
\newblock {\em The Annals of Statistics}, 43(6):2588--2623.

\bibitem[Bao and He, 2023]{BaoH23Q}
Bao, Z. and He, Y. (2023).
\newblock Quantitative {CLT} for linear eigenvalue statistics of {Wigner} matrices.
\newblock {\em The Annals of Applied Probability}, 33(6B):5171--5207.

\bibitem[Bao et~al., 2024]{BaoH24S}
Bao, Z., Hu, J., Xu, X., and Zhang, X. (2024).
\newblock Spectral statistics of sample block correlation matrices.
\newblock {\em The Annals of Statistics}, 52(5):1873--1898.

\bibitem[Bao and George, 2024]{BaoG24U}
Bao, Z. and George, D.~M. (2024).
\newblock Ultra high order cumulants and quantitative CLT for polynomials in
random matrices.
\newblock  arxiv preprint arXiv:2411.11341.

\bibitem[Berezin and Bufetov, 2021]{BerezinB21R}
Berezin, S. and Bufetov, A.~I. (2021).
\newblock On the rate of convergence in the central limit theorem for linear statistics of {Gaussian}, {Laguerre}, and {Jacobi} ensembles.
\newblock {\em Pure and Applied Functional Analysis}, 6(1):57--99.

\bibitem[Billingsley, 1995]{Billingsley95}
Billingsley, P. (1995).
\newblock {\em Probability and Measrue}, 3rd ed. Wiley, New York.

\bibitem[Burkholder, 1973]{Burkholder73D}
Burkholder, D. L. (1973).
\newblock Distribution function inequalities for martingales.
\newblock {\em The Annals of Probability}, 1(1):19--42.

\bibitem[Chen, 1981]{Chen81B}
Chen, X. (1981).
\newblock {Berry-Esseen} bounds for error variance estimates in linear models.
\newblock {\em Scientia Sinica}, 24(7):899--913.

\bibitem[Cui et~al., 2025]{CuiH25R}
Cui, J., Hu, J., Bai, Z., and Hu, G. (2025).
\newblock On the rate of convergence in the CLT for LSS of large-dimensional sample covariance matrices.
\newblock  arxiv preprint arXiv:2506.02880.

\bibitem[Diaconis and Shahshahani, 1994]{DiaconisS94E}
Diaconis, P. and Shahshahani, M. (1994).
\newblock On the eigenvalues of random matrices.
\newblock {\em Journal of Applied Probability}, 31:49--62.

\bibitem[Ding and Wang, 2023]{DingW25G}
Ding, X. and Wang, Z. (2023).
\newblock Global and local {CLT}s for linear spectral statistics of general sample covariance matrices when the dimension is much larger than the sample size with applications.
\newblock arxiv preprint arXiv:2308.08646.

\bibitem[D{\"o}bler and Stolz, 2014]{DoblerS14Q}
D{\"o}bler, C. and Stolz, M. (2014).
\newblock A quantitative central limit theorem for linear statistics of random matrix eigenvalues.
\newblock {\em Journal of Theoretical Probability}, 27(3):945--953.

\bibitem[Hu et~al., 2019]{HuL19H}
Hu, J., Li, W., Liu, Z., and Zhou, W. (2019).
\newblock High-dimensional covariance matrices in elliptical distributions with application to spherical test.
\newblock {\em The Annals of Statistics}, 47(1):527--555.

\bibitem[Johansson, 1997]{Johansson97R}
Johansson, K. (1997).
\newblock On random matrices from the compact classical groups.
\newblock {\em Annals of Mathematics}, 145(3):519--545.

\bibitem[Jonsson, 1982]{Jonsson82Lb}
Jonsson, D. (1982).
\newblock Some limit theorems for the eigenvalues of a sample covariance matrix.
\newblock {\em Journal of Multivariate Analysis}, 12(1):1--38.

\bibitem[Lambert et~al., 2019]{LambertL19Q}
Lambert, G., Ledoux, M., and Webb, C. (2019).
\newblock Quantitative normal approximation of linear statistics of $\beta$--ensembles.
\newblock {\em The Annals of Probability}, 47(5):2619--2658.

\bibitem[Li et~al., 2016]{LiB16C}
Li, H., Bai, Z., and Hu, J. (2016).
\newblock Convergence of empirical spectral distributions of large dimensional quaternion sample covariance matrices.
\newblock {\em Annals of the Institute of Statistical Mathematics}, 68(4):765--785.

\bibitem[Liu et~al., 2023]{LiuH23C}
Liu, Z., Hu, J., Bai, Z., and Song, H. (2023).
\newblock A {CLT} for the {LSS} of large-dimensional sample covariance matrices with diverging spikes.
\newblock {\em The Annals of Statistics}, 51(5):2246--2271.

\bibitem[Lytova and Pastur, 2009]{LytovaP09C}
Lytova, A. and Pastur, L. (2009).
\newblock Central limit theorem for linear eigenvalue statistics of random matrices with independent entries.
\newblock {\em The Annals of Probability}, 37(5):1778--1840.

\bibitem[Mar{\v c}enko and Pastur, 1967]{MarcenkoP67D}
Mar{\v c}enko, V. A. and Pastur, L. A. (1967).
\newblock Distribution of eigenvalues for some sets of random matrices. 
\newblock {\em Mathematics of the USSR-Sbornik}, 1(4):457--483.

\bibitem[Najim and Yao, 2016]{NajimY16G}
Najim, J. and Yao, J. (2016).
\newblock Gaussian fluctuations for linear spectral statistics of large random covariance matrices.
\newblock {\em The Annals of Applied Probability}, 26(3):1837--1887. 

\bibitem[Pan and Zhou, 2008]{PanZ08C}
Pan, G. and Zhou, W. (2008).
\newblock Central limit theorem for signal-to-interference ratio of reduced rank linear receiver.
\newblock {\em The Annals of Applied Probability}, 18(3):1232--1270.

\bibitem[Schnelli and Xu, 2023]{SchnelliX23C}
Schnelli, K. and Xu, Y. (2023).
\newblock Convergence rate to the {T}racy--{W}idom laws for the largest eigenvalue of sample covariance matrices.
\newblock {\em The Annals of Applied Probability}, 33(1):677--725.

\bibitem[Shcherbina, 2011]{Shcherbina11C}
Shcherbina, M. (2011).
\newblock Central limit theorem for linear eigenvalue statistics of the {W}igner and sample covariance random matrices.
\newblock {\em Zhurnal Matematicheskoi Fiziki, Analiza, Geometrii}, 7(2):176--192.

\bibitem[Silverstein, 1995]{Silverstein95S}
Silverstein, J.~W. (1995).
\newblock Strong convergence of the empirical distribution of eigenvalues of large dimensional random matrices.
\newblock {\em Journal of Multivariate Analysis}, 55(2):331--339.

\bibitem[Silverstein and Choi, 1995]{SilversteinC95A}
Silverstein, J.~W. and Choi, S.~I. (1995).
\newblock Analysis of the limiting spectral distribution of large dimensional random matrices.
\newblock {\em Journal of Multivariate Analysis}, 54(2):295--309.

\bibitem[Yin et~al., 1988]{YinB88L}
Yin, Y., Bai, Z., and Krishnaiah, P.~R. (1988).
\newblock On the limit of the largest eigenvalue of the large dimensional sample covariance matrix.
\newblock {\em Probability Theory and Related Fields}, 78(4):509--521.

\bibitem[Zheng et~al., 2015]{ZhengB15S}
Zheng, S., Bai, Z., and Yao, J. (2015).
\newblock Substitution principle for {CLT} of linear spectral statistics of high-dimensional sample covariance matrices with applications to hypothesis testing.
\newblock {\em The Annals of Statistics}, 43(2):546--591.
\end{thebibliography}
\end{document}